\documentclass[12pt]{article} 

\usepackage{latexsym,amsmath,amssymb,a4wide,amscd}
\usepackage[latin1]{inputenc}
\usepackage[T1]{fontenc}

\usepackage{array}

\usepackage{amsthm}

\usepackage{lscape}

\usepackage{url}

\usepackage{longtable}
\usepackage{multirow}

\usepackage{MnSymbol}

\usepackage{amsthm}
\numberwithin{equation}{section}
\numberwithin{figure}{section}
\theoremstyle{plain}
\newtheorem {satz}{Theorem}[section]
\newtheorem {lemma}[satz]{Lemma}
\newtheorem {prop}[satz]{Proposition}
\newtheorem {kor}[satz]{Corollary}
\newtheorem {conj}[satz]{Conjecture}

\theoremstyle{definition}
\newtheorem {deff}[satz]{Definition}
\theoremstyle{remark}
\newtheorem {remark}[satz]{Remark}

\usepackage[pdftex]{graphicx} 

\usepackage[nooneline]{caption}

\begin{document}

\date{}

\title{\Large {\bf Combinatorial $3$-manifolds with transitive cyclic symmetry}}

\author{Jonathan Spreer}

\maketitle

\subsection*{\centering Abstract}

{\em
	In this article we give combinatorial criteria to decide whether a transitive cyclic combinatorial $d$-manifold can be generalized to an infinite family of such complexes, together with an explicit construction in the case that such a family exists. In addition, we substantially extend the classification of combinatorial $3$-manifolds with transitive cyclic symmetry up to $22$ vertices. Finally, a combination of these results is used to describe new infinite families of transitive cyclic combinatorial manifolds and in particular a family of neighborly combinatorial lens spaces of infinitely many distinct topological types. 
}\\
\\
\textbf{MSC 2010: }
57Q15;  
57N10; 
57M05 
\\
\textbf{Keywords: } combinatorial $3$-manifold, transitive cyclic symmetry, transitive automorphism group, fundamental group, simplicial complexes, difference cycles, lens spaces

\section{Introduction}

In combinatorial topology, the principal objects of study are {\em combinatorial manifolds} which are abstract simplicial complexes together with additional local constraints. In this way, topological information is encoded purely combinatorially making the field of topology accesible to algorithmic methods. As a consequence, one topological manifold has a large number of combinatorially distinct presentations and attention has to be paid to the choice of a suitable combinatorial manifold. One way to find combinatorial manifolds which are easy-to-handle and easy to analyze is to reduce a given combinatorial manifold in the number of its simplices using {\em bistellar moves} (see \cite{Bjoerner00SimplMnfBistellarFlips}), another method is to construct highly symmetric combinatorial manifolds which can be described efficiently just by the generators of its symmetry and a system of orbit representatives of the simplices.

Here we will focus on the latter technique. More precisely, we will look at $n$-vertex combinatorial manifolds which do not change under a vertex-shift of type $v \mapsto v+1 \mod n$. Such combinatorial manifolds will be called {\em cyclic}. These objects are interesting to investigate for several reasons. They account for the largest group of so-called {\em transitive combinatorial manifolds}, i.e. combinatorial manifolds which are globally defined by a local neighborhood of a single vertex. Cyclic combinatorial manifolds can be described using so-called {\em difference cycles} (cf. Definition \ref{def:diffCycle}), they are especially easy to analyze and easy to work with. Because of their easy combinatorial structure they allow theoretical proofs on a scale where non-symmetric combinatorial manifolds need complicated computer proofs or cannot be handled at all and other highly symmteric combinatorial manifolds fail to be easy to investigate because of the complicated structure of their symmetry.

As a consequence, cyclic combinatorial manifolds have been used to establish upper and lower bounds for combinatorial properties of simplicial complexes (see \cite{Kuhnel88CombDToriLargeSymm} for small triangulations of the $d$-torus and \cite{Kuehnel96PermDiffCyc} for a set of several infinite families of combinatorial manifolds due to K\"uhnel and Lassmann), to prove tightness of existing upper and lower bounds (see \cite[Theorem 5.5]{Kuehnel95TightPolySubm} for an infinite family of tight and minimal sphere bundles over the circle), and to provide example complexes that are easy to describe, efficient to store and easy to work with (see \cite[Theorem 4]{Altshuler71PolyhedralRealizationsTori} for the infinite family of the so-called Altshuler tori with dihedral automorphism group, \cite{Kuehnel85NeighbComb3MfldsDihedralAutGroup} for a $2$-neighborly\footnote{A simplicial complex is called {\em $k$-neighborly} if every $k$-tuple of vertices spans a face of the complex} infinite family of the $3$-dimensional Klein bottle, and \cite{Brehm09LatticeTrigE33Torus} for a $2$-neighborly infinite family of the $3$-torus). Altogether, cyclic combinatorial manifolds are a natural choice when looking for a combinatorial version of a topological manifold to work with. 

In addition, most of the applications mentioned not only need single cyclic combinatorial manifolds but infinite families of such complexes to prove their respective results. Thus, providing construction principles for such families is of particular interest. In Section \ref{sec:inf}, we will present a simple combinatorial criterion to decide if a cyclic $d$-dimensional combinatorial manifold can be modified in its description to give rise to an infinite number of cyclic $d$-dimensional combinatorial manifolds. In dimension $d=3$, we explain a similar condition which detects {\em all} cyclic combinatorial $3$-manifolds that can be generalized in this way.

Furthermore, in Section \ref{sec:class} we classify all cyclic combinatorial $3$-manifolds with up to $22$ vertices. Thus, we significantly extend the number of triangulations where the findings from Section \ref{sec:inf} can be applied in order to obtain further insights into the world of combinatorial $3$-manifolds. The classification provides small and easy-to-handle combinatorial manifolds of $67$ distinct topological $3$-manifolds which can be used as a canonical choice to work with these particular topological manifolds on a combinatorial basis.

The classification itself is only the latest one in a line of other classifications dating back almost $30$ years. Neighborly combinatorial $3$-manifolds with dihedral automorphism group with up to $19$ vertices as well as cyclic neighborly combinatorial $3$-manifolds with up to $15$ vertices have already been classified by K\"uhnel and Lassmann in 1985, see \cite{Kuehnel85NeighbComb3MfldsDihedralAutGroup}. In 1999, a more general classification of all transitive combinatorial manifolds with up to $13$ vertices was presented by Lutz in \cite{Lutz03TrigMnfFewVertVertTrans} (which also contains a classification of all transitive combinatorial $d$-manifolds up to $15$ vertices in the cases $d \leq 3$ and $d \geq 9$). More recently, Lutz extended the classification of transitive combinatorial $2$-manifolds up to $21$ vertices (cf. \cite{Lutz09EquivdCovTrigsSurf}) and the classification of transitive combinatorial $3$-manifolds up to $17$ vertices (cf. \cite{Lutz11TrigMnflds}). These classification results have been proved to be useful as an extended set of small example triangulations with interesting properties \cite{Lutz08ManifoldPage}, to obtain additional insights into upper bounds on minimal triangulations of $3$-manifolds \cite{Lutz08FVec3Mnf} and as a starting point to find infinite families of combinatorial manifolds \cite{Kuehnel85NeighbComb3MfldsDihedralAutGroup}.

Finally, the classification together with the results from Section \ref{sec:inf} sets the ground work for many more results of similar nature. Most of this work is still in progress and will be presented separately in \cite{Spreer12VarCyclicPolytopeCompExpI}. As a preview of this work we present an infinite family of cyclic combinatorial $3$-manifolds of pairwise topologically distinct lens spaces in Section \ref{sec:lensspace}. 

\begin{remark}
	\label{rem:infFam}
	At this point it is important to stress that there is no canonical notion of an {\em infinite family of cyclic combinatorial manifolds}: any infinite set of cyclic combinatorial manifolds defines such a family. However, in order to be useful for applications such a family should be i) easy to describe and ii) easy to generate, i.e. given in precise terms of its combinatorial structure. In order to meet these requirements, we will focus on generalizations of cyclic combinatorial manifolds via arithmetic progressions for the entries of the orbit representatives of the respective complexes and will refer to the resulting infinite sets of cyclic combinatorial manifolds as {\em infinte families}.
\end{remark}

\section{Preliminaries}
\label{sec:prelims}

An {\em abstract simplicial complex}, $C$, can be seen as a combinatorial structure consisting of tuples $\langle a_0 , a_1, \ldots , a_d \rangle$, $a_i \in \mathbb{Z}_n$, $0 \leq i \leq d$, where the $n$ elements of $\mathbb{Z}_n$ are referred to as the vertices of the complex (cf.\ \cite{Kuehnel96PermDiffCyc}). The {\em automorphism group}, $\operatorname{Aut}(C)$, of $C$ is the group of all permutations $\sigma \in S_n$ of the vertices of $C$ which do not change $C$ as a whole. If $\operatorname{Aut}(C)$ acts transitively on the vertices, $C$ is called a {\em transitive} simplicial complex. If a transitive simplicial complex is invariant under the cyclic $\mathbb{Z}_n$-action $v \mapsto v+1 \mod n$, i.\ ~e.\ if for a complex $C$, possibly after a relabeling of the vertices, $\mathbb{Z}_n = \langle ( 0, 1, \ldots , n-1 ) \rangle$ is a subgroup of $\operatorname{Aut}(C)$ then $C$ is called a {\em cyclic} simplicial complex or a complex with {\em transitive cyclic symmetry}. 

For cyclic simplicial complexes we have the following situation: Since the whole complex does not change under a vertex-shift of type $v \mapsto v+1 \mod n$, two tuples are in one orbit of the cyclic group action if and only if the differences modulo $n$ of its vertices are equal. Hence, we can compute a system of orbit representatives by just looking at the differences modulo $n$ of the vertices of all tuples of the simplicial complex. This motivates the following definition.

\begin{deff}[Difference cycle]
	\label{def:diffCycle}
	Let $a_i$, $0 \leq i \leq d$, be positive integers such that $ n := \sum_{i=0}^{d} a_i$. The simplicial complex
	\begin{equation*}
		( a_0 : \ldots : a_{d} ) := \mathbb{Z}_n \langle 0 , a_0 , \ldots , \Sigma_{i=0}^{d-1} a_i \rangle
	\end{equation*}
	is called {\em difference cycle of dimension $d$ on $n$ vertices} where $G \langle \cdot \rangle$ denotes the $G$-orbit of $\langle \cdot \rangle$. The number of elements of $(a_0 : \ldots : a_d)$ is referred to as the {\em length} of the difference cycle and a difference cycle of length $n$ is said to be of {\em full length}.

	If a simplicial complex $C$ is a union of difference cycles of dimension $d$ on $n$ vertices and $\lambda$ is a unit of $\mathbb{Z}_n$ such that the complex $\lambda C$ (obtained by multiplying all vertex labels by $\lambda$ modulo $n$) equals $C$, then $\lambda$ is called a {\em multiplier} of $C$.
\end{deff}

Note that for any unit $\lambda \in \mathbb{Z}_n^{\times}$, the complex $\lambda C$ is combinatorially isomorphic to $C$. In particular, all $\lambda \in \mathbb{Z}_n^{\times}$ are multipliers of the complex $\bigcup_{\lambda \in \mathbb{Z}_n^{\times}} \lambda C$ by construction.  The definition of a difference cycle above is equivalent to the one given in \cite{Kuehnel96PermDiffCyc}. 

Difference cycles as well as other cyclic combinatorial structures have been thoroughly investigated under purely combinatorial aspects (see for example \cite[Part V]{Lindner80CyclicSteinerSystems} for a work on cyclic Steiner systems in the field of design theory, see \cite[Section III.4]{Emch29TripleQuadrSystems} for the highly symmetric {\em Emch design}). However, interpreting these well-known combinatorial structures as simplicial complexes and hence geometric objects is a relatively new development (see \cite{Kuehnel98TopAsp2foldTripSys} for an interpretation of two-fold triple systems as singular surfaces or \cite[Section 5.4]{Spreer10Diss} for an interpretation of Emch's design as the quotient of a hyperbolic tesselation of $3$-space).

\medskip
A {\em combinatorial manifold} is a special kind of simplicial complex which is defined as follows: An abstract simplicial complex $M$ is said to be {\em pure}, if all of its tuples are of length $d+1$, where $d$ is referred to as the {\em dimension} of $M$. If, in addition, any vertex link of $M$, i.e. the boundary of a simplicial neighborhood of a vertex of $M$, is a triangulated $(d-1)$-sphere endowed with the standard piecewise linear structure, $M$ is called a {\em combinatorial $d$-manifold}. Throughout this article, we will describe {\em cyclic combinatorial manifolds} as a set of difference cycles. In this way, many problems dealing with cyclic combinatorial manifolds can be solved in an elegant way.

One of the principal tools to analyze combinatorial manifolds is the use of a discrete Morse type theory following Kuiper, Banchoff and K\"uhnel \cite{Kuiper71MorseRelations,Banchoff67CritPntCurvEmbPoly,Banchoff83CritPtsCurvEmbPolyhII,Kuehnel95TightPolySubm}. In this theory, the discrete analogue of a Morse function is given by a mapping from the set of vertices $V$ of a combinatorial manifold $M$ to the real numbers $\mathbb{R}$ such that no two vertices have the same image, in this way inducing a total ordering on $V$. This mapping can then be extended to a function $ f: M \to \mathbb{R} $  by linearly interpolating the values of the vertices of a face of $M$ for all points inside that face. $f$ is called a {\em regular simplexwise linear function} or {\em rsl-function} on $M$.

A point $x \in M$ is said to be {\em critical} for an rsl-function $f:M \to \mathbb{R}$ if \[H_{\star} (M_x , M_x \backslash \{ x \} , \mathbb{F}) \neq 0 \] where $M_x := \{ y \in M \, | \, f(y) \leq f(x) \}$ and $\mathbb{F}$ is a field. Here, $H_{\star}$ denotes an appropriate homology theory. It follows that no point of $M$ can be critical except possibly the vertices. More precisely we call a vertex $v$ {\em critical of index $i$ and multiplicity $m$} if $\beta_i (M_v , M_v \backslash \{ v \} , \mathbb{F}) = m$. 

A result of Kuiper \cite{Kuiper71MorseRelations} states that the number of critical points of an rsl-function of $M$ counted by multiplicity is an upper bound for the sum of the Betti numbers of $M$, hence extending the famous {\em Morse relations} from the smooth theory to the discrete setting.

The pre-image of a point under an rsl-function which does not meet any vertex of the surrounding combinatorial manifold is called a {\em slicing}. By construction, a slicing is an embedded co-dimension $1$ submanifold which contains information about the topology of the surrounding manifold (see Figures \ref{fig:51}, \ref{fig:homSphere} and \ref{fig:HeegaardDiagram} for slicings in the case $d=3$, and \cite{Spreer10NormSurfsCombSlic} for further details about slicings).

\section{Infinite families of transitive cyclic combinatorial manifolds}
\label{sec:inf}

For cyclic combinatorial manifolds, one straightforward generalization of a given complex to an {\em infinite family of cyclic combinatorial manifolds} with increasing number of vertices can be constructed by finding a slicing of a cyclic combinatorial manifold which exhibits the symmetry of the surrounding manifold in a cyclic pattern, extending this pattern and then re-constructing a larger version of the original combinatorial manifold (cf. \cite[Section 4.5]{Spreer10Diss}). 

However, in this article we want to focus on when a combinatorial $d$-manifold given by a set of difference cycles can be generalized using arithmetic progressions for the entries of its difference cycles (cf. Remark \ref{rem:infFam}). More precisely, for a combinatorial manifold $M = \{ d_1 , \ldots , d_m \}$ with $n$ vertices represented by $m$ difference cycles $d_i = ( a_i^0 : \ldots : a_i^{d} )$, $1 \leq i \leq m$, we define complexes $M_k = \{ d_{1,k} , \ldots , d_{m,k} \}$, $k > - \min_{1 \leq i \leq m} \{ a_i^{d} \}$, with $n+k$ vertices and difference cycles $d_{i,k} = ( a_i^0 : \ldots : a_i^{d-1} : a_i^{d} + k )$, $1 \leq i \leq m$, and ask for a purely combinatorial condition on $M$ to check whether (and for which range of $k$) the complexes $M_k$ are combinatorial manifolds (see Theorem \ref{order2} below). We will start with discussing the case $d=3$, where the ``generalizability'' of a cyclic combinatorial manifold is even equivalent to the following combinatorial criterion.

\begin{satz}
	\label{main}
	Let $M = \{ d_1 , \ldots , d_m \}$ be a combinatorial $3$-manifold with $n$ vertices, represented by $m$ difference cycles $d_i = ( a_i^0 : a_i^1 : a_i^2 : a_i^{3} )$, $1 \leq i \leq m$. Then the complex $M_k$ is a combinatorial manifold for all $k \geq 0$ if and only if $a_i^{3} > \frac{n}{2}$ for all  $1 \leq i \leq m$.
\end{satz}

In order to prove Theorem \ref{main} let us first take a look at a few lemmas.

\begin{lemma}
	\label{lengthOfDC}
	Let $( a_0 : \ldots : a_{d} )$ be a difference cycle of dimension $d$ on $n$ vertices and $1 \leq k \leq d+1$ the smallest integer such that $k \mid (d+1)$ and $a_i = a_{i+k}$, $0 \leq i \leq d-k$. Then $( a_0 : \ldots : a_{d} )$ has length $\frac{nk}{d+1}.$
\end{lemma}

\begin{proof}
	Since $\sum_{i=0}^{d} a_i = n$ and $a_i = a_{i+k}$, $0 \leq i \leq d-k$, we have $\sum_{i=0}^{d} a_i = \frac{d+1}{k} \sum_{i=0}^{k-1} a_i$ and hence $\sum_{i=0}^{k-1} a_i = \frac{nk}{d+1}$. Keeping this in mind we have (all entries are computed modulo $n$)
	\small
	\begin{eqnarray}
		\left \langle 0 + \tfrac{nk}{d+1}, a_0 + \tfrac{nk}{d+1}, \ldots , (\Sigma_{i=0}^{d-1} a_i) + \tfrac{nk}{d+1} \right \rangle &=& \left \langle \Sigma_{i=0}^{k-1} a_i , \Sigma_{i=0}^{k} a_i, \ldots , \Sigma_{i=0}^{d-1} a_i , 0 , a_0 , \ldots , \Sigma_{i=0}^{k-2} a_i \right \rangle \nonumber \\
	 &=& \left \langle 0 , a_0, \ldots , \Sigma_{i=0}^{d-1} a_i \right \rangle . \nonumber
	\end{eqnarray}
	\normalsize
	Hence, for the length $\ell$ of $( a_0 : \ldots : a_{d} )$ we have $\ell \leq \frac{nk}{d+1}$ and since $k$ is minimal with $k \mid (d+1)$ and $a_i = a_{i+k}$, the upper bound is attained.
\end{proof}

\begin{lemma}
	\label{link}
	Let $M_k$, $k \geq 0$, be an infinite family of cyclic combinatorial $3$-manifolds with $n+k$ vertices represented by the union of $m$ difference cycles of full length. Then we have for the $f$-vectors of the vertex links
	$$f \left (\operatorname{lk}_{M_0} (0)) \right ) = f \left (\operatorname{lk}_{M_k} (0)) \right ) = (2m+2, 6m, 4m) $$ 
	for all $k \geq 0$. In particular, the number of vertices of $\operatorname{lk}_{M_k} (0)$ does not depend on the value of $k$.	
\end{lemma}

\begin{proof}
	Since $M_k$ is the union of $m$ difference cycles of full length, we have for the number of tetrahedra $f_{3} (M_k) = m(n+k)$ for all $k \geq 0$. Furthermore, as $M_k$ is cyclic, all vertices are contained in the same number of tetrahedra which have $4$ vertices. By the fact that any facet of $\operatorname{lk}_{M_k} (0)$ corresponds to a facet in $M_k$ containing $0$ it follows that for the number of triangles of the link $f_{2} (\operatorname{lk}_{M_k} (0)) = \frac{4m(n+k)}{n+k} = 4m$ holds, which is independent of $k$. Since for all $k \geq 0$ the link $\operatorname{lk}_{M_k} (0)$ is a combinatorial $2$-sphere, all edges of $\operatorname{lk}_{M_k} (0)$ lie in exactly two triangles, hence $f_{1} (\operatorname{lk}_{M_k} (0)) = 6m$. Finally, the Euler characteristic of the $2$-sphere is $2$, and by the Euler-Poincar\'e formula we have $f_{0} (\operatorname{lk}_{M_k} (0)) = 2m+2$.	
\end{proof}

Let us now come to the proof of Theorem \ref{main}.

\begin{proof}
	Let $M = \{ d_1 , \ldots , d_m \}$ be a combinatorial $3$-manifold with $n$ vertices, represented by $m$ difference cycles $d_i = ( a_i^0 : a_i^1 : a_i^2 : a_i^{3} )$, $1 \leq i \leq m$, such that $a_i^{3} > \frac{n}{2} > a_i^0 + a_i^1 + a_i^{2}$ for all  $1 \leq i \leq m$.
	For the link of vertex $0$ in $M$ we then have (all entries are computed modulo $n$)
	\begin{multline}
		\label{eq:sphere}
		\operatorname{lk}_M (0) \,\, = \,\, \bigcup \limits_{i=1}^{m} \,\, \left \{ 
			\langle a_i^0, a_i^0 + a_i^1, a_i^0 + a_i^1 + a_i^2 \rangle, \langle - a_i^0, a_i^1, a_i^1 + a_i^2 \rangle, \langle - a_i^0 - a_i^1 , - a_i^1 , a_i^2 \rangle, 
			\right . \\ \left . \langle - a_i^0 - a_i^1 - a_i^2 , - a_i^1 - a_i^2 , - a_i^2 \rangle \right \},
	\end{multline}
	and for the link of vertex $0$ in $M_k$ (all entries are computed modulo $n+k$)
	\begin{multline}
		\label{eq:maybeSphere}
		\operatorname{lk}_{M_k} (0) \,\, = \,\, \bigcup \limits_{i=1}^{m} \,\, \left \{ 
			\langle a_i^0, a_i^0 + a_i^1, a_i^0 + a_i^1 + a_i^2 \rangle, \langle - a_i^0, a_i^1, a_i^1 + a_i^2 \rangle, \langle - a_i^0 - a_i^1 , - a_i^1 , a_i^2 \rangle, 
			\right . \\ \left . \langle - a_i^0 - a_i^1 - a_i^2 , - a_i^1 - a_i^2 , - a_i^2 \rangle \right \}.
	\end{multline}
	Since $M$ is a combinatorial $3$-manifold, (\ref{eq:sphere}) must be a triangulated $2$-sphere. Since $a_i^{3} > \frac{n}{2} > a_i^0 + a_i^1 + a_i^{2}$ for all $1 \leq i \leq m$, the vertices $v_j$ of $\operatorname{lk}_M (0)$ can be mapped to the vertices of $\operatorname{lk}_{M_k} (0)$, $k \geq 0$, as follows:
	$$ v_j \mapsto \left\{ \begin{array}{ll} v_j & \textrm{ if } v_j < \frac{n}{2} \\ v_j + k & \textrm{ if } v_j \geq \frac{n}{2} . \end{array} \right.$$
	Applying this relabeling to the vertices of $M$ yields a simplicial complex on vertices of $M_k$ equal to (\ref{eq:maybeSphere}) and hence a combinatorial isomorphism between $\operatorname{lk}_M (0)$ and $\operatorname{lk}_{M_k} (0)$. Since $M$ and $M_k$ are cyclic, all vertex links are isomorphic. Altogether it follows that $M_k$ is a combinatorial manifold for all $k \geq 0$. 
	
	This part of the proof can be generalized to combinatorial $d$-manifolds for arbitrary $d$, see Theorem \ref{order2}.  
	
	\medskip
	Conversely, let $M = \{ d_1 , \ldots , d_m \}$ contain a difference cycle $d_i = ( a_i^0 : a_i^1 : a_i^2 : a_i^{3} )$, $1 \leq i \leq m$, such that $a_i^{3} \leq \frac{n}{2}$ and let $\tilde{k} := \underset{1 \leq j \leq m}{\operatorname{max}}( a_j^0 + a_j^1 + a_j^{2} - a_j^3 )$. Since by construction $a_j^3 + \tilde{k} \geq a_j^0 + a_j^1 + a_j^{2}$ and $a_j^l > 0$ for all $1 \leq j \leq m$, $0 \leq l \leq 3$, it follows by Lemma \ref{lengthOfDC} that all difference cycles of $M_{\tilde{k}}$ and $M_{\tilde{k}+1}$ have full length. By Lemma \ref{link} it now follows that the links of vertex $0$ in $M_{\tilde{k}}$ and $M_{\tilde{k}+1}$ have the same $f$-vector. On the other hand, since $a_i^{3} + \tilde{k} = a_i^0 + a_i^1 + a_i^{2}$ for at least one $1 \leq i \leq m$ and $a_j^{3} + \tilde{k} \geq a_j^0 + a_j^1 + a_j^{2}$ for all $1 \leq j \leq m$, $\operatorname{lk}_{M_{\tilde{k}+1}}(0)$ has to have strictly more vertices than the link of vertex $0$ in $M_{\tilde{k}}$: Namely, we have $- \Sigma_{r=0}^{2} a_i^{r} = \Sigma_{r=0}^{2} a_i^{r}$ for at least one $1 \leq i \leq m$ in $\operatorname{lk}_{M_{\tilde{k}}}(0)$, and in $\operatorname{lk}_{M_{\tilde{k}+1}}(0)$ this vertex splits into two distinct vertices. On the other hand, two vertices which are distinct in $\operatorname{lk}_{M_{\tilde{k}}}(0)$ cannot be equal in $\operatorname{lk}_{M_{\tilde{k}+1}}(0)$ by $a_j^{3} + \tilde{k} \geq a_j^0 + a_j^1 + a_j^{2}$ for all $1 \leq j \leq m$. This is a contradiction to Lemma \ref{link}. 
\end{proof}

From now on, we will require an infinite family of cyclic combinatorial manifolds to start with the smallest complex possible, that is, the complex $M_{-1}$ must not be a combinatorial manifold.

\begin{kor}
	Let $M_k$, $k \geq 0$, be an infinite family of cyclic combinatorial $3$-manifolds such that $M_{-1}$ is not a combinatorial manifold, then $M_0$ has an odd number of vertices. 
\end{kor}

\begin{proof}
	This follows immediately from the fact that $\Delta_j := a_j^3 - a_j^0 - a_j^1 - a_j^2 > 0$ for all $1 \leq j \leq m$ in $M_0$. If the mimimum over all $\Delta_j$, $1\leq j \leq m$, is strictly greater than $1$, the vertices of $\operatorname{lk}_{M_0}(0)$ are either strictly smaller than $\lfloor \frac{n}{2} \rfloor$ or strictly greater than $\lceil \frac{n}{2} \rceil$. Following the proof of Theorem \ref{main}, the mapping
$$ v_j \mapsto \left\{ \begin{array}{ll} v_j & \textrm{ if } v_j < \frac{n}{2} \\ v_j - 1 & \textrm{ if } v_j \geq \frac{n}{2} + 1 . \end{array} \right.$$
of the vertices of $\operatorname{lk}_{M_0}(0)$ to the vertices of $\operatorname{lk}_{M_{-1}}(0)$ yields a combinatorial isomorphism between the vertex links of $M_0$ and $M_{-1}$ and hence $M_{-1}$ is a combinatorial $3$-manifold. Hence, $\Delta_i = 1$ for some $1 \leq i \leq m$ and $n = 2 a_i^{3} - 1$. 
\end{proof}

\medskip
So far, we have only considered one particular type of arithmetic progression which led to infinite families of cyclic combinatorial manifolds that have members for all integers $n \geq n_0$ for $n_0$ sufficiently large. A more general but closely related approach results in other (weaker) formulations of infinite families of cyclic combinatorial manifolds: Let $N = \{ d_1 , \ldots , d_m \}$ be a combinatorial $d$-manifold with $n$ vertices, represented by $m$ difference cycles $d_i = ( a_i^0 : \ldots : a_i^{d} )$, $1 \leq i \leq m$. The simplicial complexes $N_k = \{ d_{1,k}, \ldots , d_{m,k} \}$, $k \geq 0$, with $n + \ell k$ vertices, $\ell \geq 1$ fixed, given by $d_{i,k} = ( a_{i}^0 + \ell_{i}^0 k  : \ldots : a_{i}^{d} + \ell_{i}^{d} k )$, $1 \leq i \leq m$, where for each $1 \leq i \leq m$ we have $\sum_{j=0}^{d} \ell_{i}^{j} = \ell$, $\ell_i^j \geq 0$, will be called an {\em infinite family of cyclic combinatorial manifolds of order $\ell$} if all $N_k$, $k\geq 0$, are combinatorial manifolds. The case $\ell=1$ coincides with the previously described type of infinite family which from now on will be referred to as a {\em dense} family.

\medskip
There is an analogue of the ``if''-part of Theorem \ref{main} for infinite families of combinatorial $d$-manifolds of order $\ell$ which can be formulated as follows.

\begin{satz}
	\label{order2}
	Let $N = \{ d_1 , \ldots , d_m \}$ be a combinatorial $d$-manifold with $n$ vertices, represented by $m$ difference cycles $d_i = ( a_i^0 : \ldots : a_i^{d} )$, $1 \leq i \leq m$.
	Then $N_k$, defined by non-negative integers $\ell \geq 1$ and $\ell_i^j$, $1 \leq i \leq m$, $0 \leq j \leq d$ with $\sum_{j=0}^{d} \ell_{i}^{j} = \ell$, $1 \leq i \leq m$, is a combinatorial $d$-manifold with $n + \ell k$ vertices for all $k \geq 0$ if
	\begin{equation}
		\label{eq:condition}
		\frac{\ell_i^j n}{\ell+1} < a_i^{j},
	\end{equation}
	holds for all $1 \leq i \leq m$, $0 \leq j \leq d$.
\end{satz}

Note that the case $d=3$ and $\ell=1$ in the following proof corresponds to the ``if''-part of Theorem \ref{main}.

\begin{proof}
	The proof is completely analogous to the one of the first part of Theorem \ref{main}. Here, too, we look at a relabeling of the vertices of the link $\operatorname{lk}_{N} (0)$ in order to transform it to $\operatorname{lk}_{N_k} (0)$ for arbitrary values of $k \geq 0$. 

	First note that from Condition (\ref{eq:condition}) we can derive 
	$$ a_i^{j} = n - \sum \limits_{p = 0, p \neq j}^{d} a_i^p < n - \sum \limits_{p = 0, p \neq j}^{d} \frac{\ell_i^p n}{\ell+1} = \frac{n}{\ell+1} \left ( \ell +1 - \sum \limits_{p = 0, p \neq j}^{d} \ell_i^p \right ) = \frac{(\ell_i^j +1)n}{\ell+1}, $$
	and hence we have
	\begin{equation}
		\label{eq:condition2}
		\frac{\ell_i^j n}{\ell+1} < a_i^{j} < \frac{(\ell_i^j +1)n}{\ell+1}.
	\end{equation}
	Now, consider the collection of disjoint intervals 
	$$\alpha_r^k := \left ] \frac{r(n+\ell k)}{\ell+1}, \frac{(r+1)(n+ \ell k)}{\ell+1} \right[ \,\, \subset \,\, ] 0 , n+ \ell k [ \, \subset \, \mathbb{R} .$$
	By Condition (\ref{eq:condition2}), each vertex of $\operatorname{lk}_{N_k} (0)$ lies in exactly one of the $\alpha_r^k$, $0 \leq r \leq \ell$. Now consider the relabeling
	$$ \iota : v_j \mapsto v_j +  \left \lfloor \frac{(\ell+1) v_j}{n} \right \rfloor k  $$
	from the vertices of $\operatorname{lk}_{N} (0)$ to the vertices of $\operatorname{lk}_{N_k} (0)$. By construction, if $v_j \in \alpha_r^0$ then $\iota (v_j) \in \alpha_r^k$ and, following the proof of Theorem \ref{main}, $\iota$ is injective and defines a combinatorial isomorphism between $\operatorname{lk}_{N} (0)$ and $\operatorname{lk}_{N_k} (0)$.
\end{proof}

Theorem \ref{order2} defines families of order $\ell$ by a purely combinatorial criterion. Since all dense families contain families of order $\ell$, the following characterisation of higher order families is interesting.

\begin{lemma}
	\label{lem:contained}
	Let $N_k = (d_{1,k} , \ldots , d_{m,k})$, $k \geq 0$, be an infinite family of combinatorial $d$-manifolds of order $\ell$, $1 \leq \ell \leq d$, with $n+\ell k$ vertices given by non-negative integers $\ell_i^j$, $1 \leq i \leq m$, $0 \leq j \leq d$, $\sum_{j=0}^{d} \ell_i^{j} = \ell$.

	If $\ell$ is a unit in $\mathbb{Z}_n$ then there exists a dense infinite family containing all but finitely many members of $N_k$, $k \geq 0$.
\end{lemma}

\begin{proof}
	Let $\ell$ be a unit in $\mathbb{Z}_n$ and let $a_{i,k}^{j}$ be the $j$-th entry of the $i$-th difference cycle of $N_k$.
	By multiplying $N_k$ by $\ell$ we get $\ell N_k = \{ (\ell a_{1,k}^{0} : \ldots : \ell a_{1,k}^{d}) , \ldots , (\ell a_{m,k}^{0} : \ldots : \ell a_{m,k}^{d}) \}$ 
	which is isomorphic to $N_k$ since if $\ell$ is a unit in $\mathbb{Z}_n$ then $\ell$ is a
	unit in $\mathbb{Z}_{n + \ell k}$, for all $k \geq 0$. Hence, we have
	\small
	\begin{eqnarray}
		\ell a_{i,k}^{j} & = & \ell (a_i^{j} + \ell_i^j k) \nonumber \\
		& = & \ell a_i^{j} + \ell \ell_i^j k \nonumber \\
		& = & \ell a_i^{j} + \ell_i^j (\ell k - (n+\ell k)) \nonumber \\
		& = & \ell a_i^{j} - \ell_i^j n \nonumber 
	\end{eqnarray}
	\normalsize
	which is independent of $k$. Now let $b_i^j := \Sigma_{r=0}^{j} \ell a_i^{r} - \ell_i^r n$ and for each $1 \leq i \leq m$ let $\pi_i$ be a permutation such that
	$$ b_{i}^{\pi_i(0)} \leq b_{i}^{\pi_i(1)} \leq \ldots \leq b_{i}^{\pi_i(d)}. $$
	It follows that for $k > b_{i}^{\pi_i(d)} - b_{i}^{\pi_i(0)} - n$ the difference cycle
	$(\tilde{a}_i^0 : \tilde{a}_i^1 : \ldots : \tilde{a}_i^{d-1} : \tilde{a}_i^{d} + \ell k)$ with $\tilde{a}_i^j := b_{i}^{\pi_i(j+1)} - b_{i}^{\pi_i(j)}$, $0 \leq j \leq d-1$,
	and $\tilde{a}_i^d = n - \Sigma_{j=0}^{d-1} \tilde{a}_i^j$ is isomorphic to $( a_{i,k}^{0} : \ldots : a_{i,k}^{d})$ for all $1 \leq i \leq m$. Note that in
	particular this means that $\tilde{a}_i^j > 0$ for all $1 \leq i \leq m$ and $0 \leq j \leq d-1$. We write
	$$\tilde{N}_{\ell k} = \{ (\tilde{a}_i^0 : \tilde{a}_i^1 : \ldots : \tilde{a}_i^{d-1} : \tilde{a}_i^{d} + \ell k) \,|\, 1 \leq i \leq m \}$$
	which by construction is isomorphic to $N_k$ and we conclude the proof by the observation that in $\tilde{N}_{\ell k}$ 
	only the $d$-th entries depend on $k$ and thus for $k_0$ sufficiently large 
	$\tilde{N}_{\ell k_0}$ satisfies the preconditions of Theorem \ref{order2}. 
	Hence $\tilde{N}_{\ell k_0}$ extends to an infinite dense family containing isomorphic copies of $N_k$ for each $k > k_0$.
\end{proof}
	
\begin{kor}
	Let $N_k$, $k \geq 0$, be an infinite family of cyclic combinatorial $d$-manifolds of order $2$ such that no dense family contains an infinite number of members of $N_k$. Then the number of vertices of $N_0$ has to be even.
\end{kor}

\begin{proof}
	If no dense family contains an infinite number of members of $N_k$, then in particular for all fixed dense families an infinite number of members of $N_k$ is not contained in this family. The statement now follows from Lemma \ref{lem:contained} since $2$ is a unit in $\mathbb{Z}_n$ for all $n \equiv 1(2)$.
\end{proof}

\begin{remark}
Theorem \ref{order2} gives us an easy-to-check combinatorial criterion for arbitrarily dimensional simplicial complexes to be combinatorial $d$-manifolds. This is of great use in dimension $d \geq 5$: Checking the manifold property of a $d$-dimensional simplicial complex involves the recognition of a $(d-1)$-dimensional sphere. While this is easy for $d \leq 3$ and there are deterministic algorithms for the case $d = 4$ \cite{Rubinstein953SphereRec}, heuristic methods have to be used in dimensions $d \geq 5$.
\end{remark}

\section{Classification of cyclic $3$-manifolds}
\label{sec:class}

There is a number of classifications containing cyclic combinatorial manifolds described in the literature due to K\"uhnel and Lassmann \cite{Kuehnel85NeighbComb3MfldsDihedralAutGroup} and Lutz in \cite{Lutz03TrigMnfFewVertVertTrans,Lutz09EquivdCovTrigsSurf,Lutz11TrigMnflds}. These classifications contain $336$ cyclic combinatorial $3$-manifolds of $13$ distinct topological types (cf. Tables \ref{fig:numbers} and \ref{fig:top} for the number of cyclic combinatorial $3$-manifolds and the number of distinct topological types up to $17$ vertices). 

The main motivation to extend this classification was to substantially extend the set of example triangulations where the results from Section \ref{sec:inf} can be applied in order to get new results. Moreover, a larger set of examples gives us a better chance to understand which topological properties are compatible with a cyclic symmetry and which ones are not.

\begin{satz}[Classification of cyclic combinatorial $3$-manifolds]
	\label{classification}
	There are $6070$ combinatorial types of (connected) combinatorial $3$-manifolds with transitive cyclic symmetry with up to $22$ vertices. These complexes split up into $67$ topological types.
\end{satz}

The number of combinatorial manifolds, combinatorial types, {\em locally minimal} combinatorial manifolds, i.e. complexes which cannot be reduced by bistellar moves without inserting new vertices, and topological types can be found in Table \ref{fig:numbers}. A list of all topological types of $3$-manifolds in the classification together with a particular combinatorial manifold of each type sorted by their model geometries is shown in Table \ref{fig:TopTypesSpherical} to \ref{fig:TopTypesOther}. An overview over all topological types of cyclic combinatorial 3-manifolds sorted by vertex number is listed in Table \ref{fig:top}. 

The calculations were done with the help of the \texttt{GAP}-package \textsf{simpcomp} \cite{simpcomp,simpcompISSAC,simpcompISSAC11} as well as \texttt{GAP} \cite{GAP4}. In addition, the $3$-manifold software packages \texttt{regina} by Burton et al. \cite{Burton09Regina}, \texttt{SnapPy} by Weeks \cite{WeeksSnapPea} and the \texttt{Three-manifold Recognizer} developed by the research group of Matveev \cite{Matveev13Recognizer} were used for topological type recognition. All cyclic manifolds are available within \textsf{simpcomp} by calling the function \texttt{SCCyclic3Mfld(n,k)} where \texttt{n} is the number of vertices and \texttt{k} is the number of a specific cyclic combinatorial $3$-manifold. The total number of cyclic \texttt{n}-vertex combinatorial $3$-manifolds can be obtained using the function \texttt{SCNrCyclic3Mfld(n)}.

\subsubsection*{Discussion of the proof}

	The proof of Theorem \ref{classification} was mainly done by computer. In the following, we will discuss the methods and software involved in the proof. 

	\medskip
	The complexes were found using the classification algorithm for transitive combinatorial manifolds due to K\"uhnel and Lassmann \cite{Kuehnel85NeighbComb3MfldsDihedralAutGroup} which is integrated into the software package \textsf{simpcomp} as of Version 1.3. The topological distinctions of most of the spherical and flat manifolds, as well as the connected sums of $S^2 \times S^1$ and $S^2 \dtimes S^1$ were done via analysis of the  fundamental group of the complexes:
	
	\begin{itemize}
		\item The manifolds of type $(S^2 \times S^1)^{\# k}$ and $(S^2 \dtimes S^1)^{\# k}$ were identified by calculating the fundamental group -- the free group on $k$ generators -- and applying Kneser's conjecture, proved by Stallings in 1959 (see \cite{Stallings59KnesersConj}) together with \cite[Theorem 5.2]{Hempel1976}.
		\item By the elliptization conjecture (stated by Thurston in \cite[Chapter 3]{Thurston02GeomTopOf3Mflds}, recently proved by Perelman, see \cite{Perelman03PC1,Perelman03PC2,Perelman03PC3}), the topological type of a spherical $3$-manifold distinct from a lens space is already determined by the isomorphism type of its (finite) fundamental group. This allows an identification of all such $3$-manifolds using the finite group recognition algorithm of \texttt{GAP}.
		\item The fundamental group distinguishes all flat $3$-manifolds by a theorem of Bieberbach (see \cite{Bieberbach12Bewegungsgruppen} and \cite[page 4]{Milnor03TowPoincareClass3Mnf}). On the other hand, all other $3$-manifolds with a fundamental group containing $\mathbb{Z}^3$ are known to be the connected sum of a flat $3$-manifold with some other $3$-manifold (cf. \cite{Luft843MfldsWithSubgroupsContainingZ3}). Hence, all $3$-manifolds with the fundamental group of a flat manifold have to be prime (as all flat manifolds are prime and the fundamental group of a $3$-manifold $M$ determines the length of a prime decomposition of $M$, cf. \cite{Stallings59KnesersConj} and \cite[Theorem 5.2]{Hempel1976}) and thus are flat. Altogether, the topological type of a $3$-manifold with the fundamental group of a flat manifold is in fact flat and the manifold is determined by its fundamental group. Hence, it can be identified by determining its fundamental group using \textsf{simpcomp} and \texttt{GAP}.
	\end{itemize}
	
	For more information about the spherical case in the classification of $3$-manifolds see \cite{Threlfall33SphaerischeMgftkten,Orlik72SeifertMflds}, for more about flat $3$-manifolds see \cite{Bieberbach12Bewegungsgruppen,Milnor03TowPoincareClass3Mnf,Kuehnel02DiffGeom}.
  
	\bigskip
	In the following we will show that only lens spaces of type $\operatorname{L}(3,1)$, $\operatorname{L}(5,1)$, $\operatorname{L}(7,1)$, $\operatorname{L}(8,3)$ and $\operatorname{L}(15,4)$ can be represented as a transitive cyclic combinatorial manifold with at most $22$ vertices.

	\begin{figure}[p]
		\begin{center}
			\includegraphics[width=0.8\textwidth]{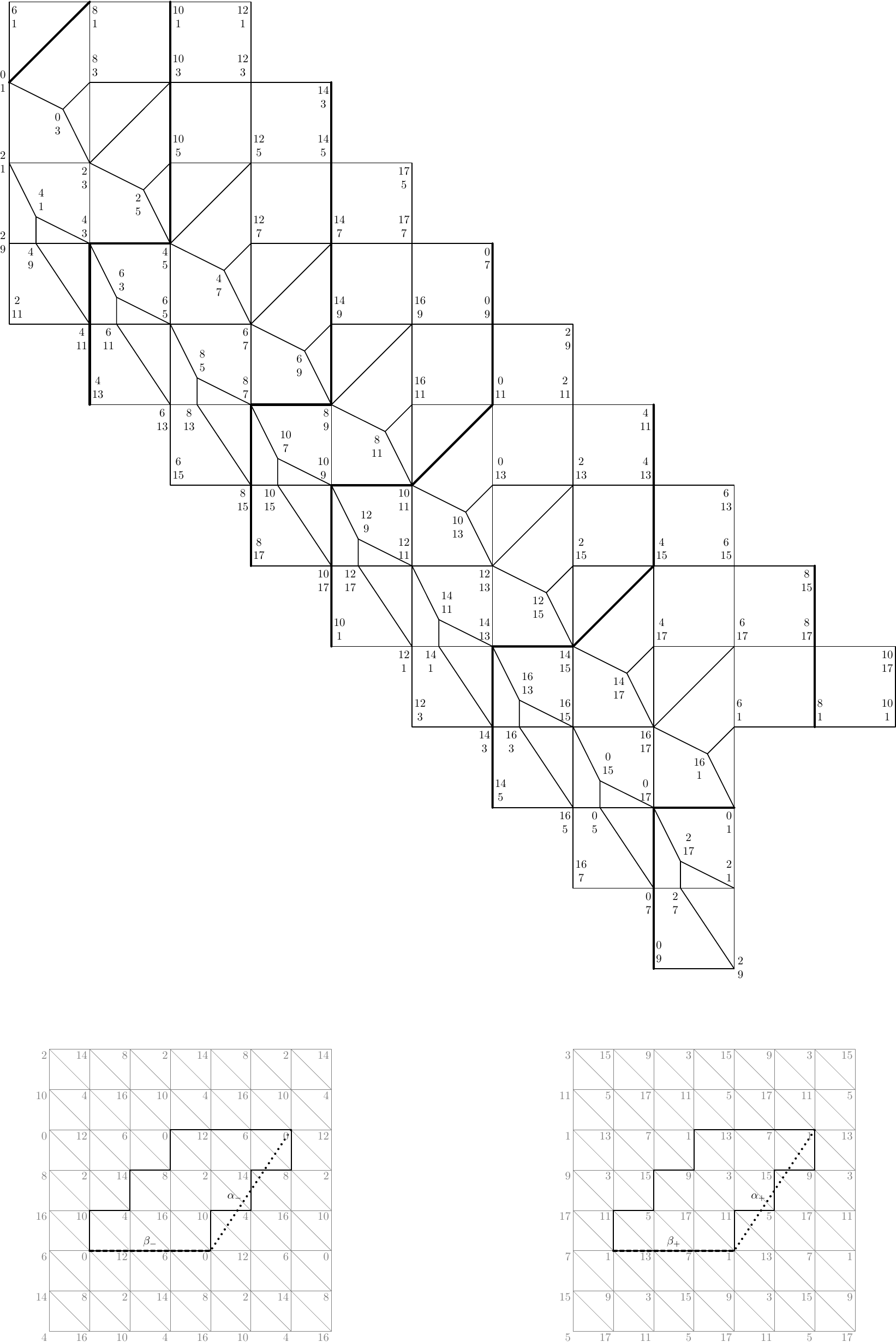}
			\caption{Slicing of $C$ between the odd labeled and the even labeled vertices together with the boundary of the two solid tori spanned by the even and by the odd vertices.\label{fig:51}}
		\end{center}
	\end{figure}

	\medskip
	First, let us show that the cyclic combinatorial manifold $C$ with $18$ vertices given by 
		\scriptsize
		$$C:=\{(1:1:1:15),(1:2:5:10),(1:5:2:10),(1:5:10:2),(2:5:2:9),(2:6:4:6),(2:7:2:7),(4:4:4:6)\}$$
		\normalsize
	is homeomorphic to the lens space $L(5,1)$.
	
	Figure \ref{fig:51} shows the slicing of $C$ between the odd labeled vertices and the even labeled vertices. Here, the slicing is a torus. Also, both the span of the odd and the span of the even labeled vertices is a solid torus and hence $C$ is a manifold of Heegaard genus at most $1$. For the $1$-homology of the two tori $T_- := \partial (\operatorname{span} (0,2, \ldots , 16))$ and $T_+ := \partial (\operatorname{span} (1,3, \ldots , 17))$ we choose a basis as follows (here and in the following $[v_0, v_1, \ldots , v_r ]$ denotes a path of edges in a simplicial complex starting at vertex $v_0$ and ending at vertex $v_r$):
	\begin{eqnarray}
		\alpha_- &:=& [ 0, 10, 4, 14, 8, 0 ] \nonumber \\
		\beta_- &:=& [ 0, 12, 6, 0 ] \nonumber
	\end{eqnarray}
	and
	\begin{eqnarray}
		\alpha_+ &:=& [ 1, 11, 5, 15, 9, 1 ] \nonumber \\
		\beta_+ &:=& [ 1, 13, 7, 1 ] \nonumber
	\end{eqnarray}
	such that $H_1 (T_{\pm}) = \langle \alpha_{\pm} , \beta_{\pm} \rangle$, $H_1 (\operatorname{span} (0,2, \ldots , 16)) = \langle \beta_{-} \rangle$ and $H_1 (\operatorname{span} (1,3, \ldots , 17)) = \langle \beta_{+} \rangle$.
	
	Now, we want to express $\alpha_-$ in terms of $\alpha_+$ and $\beta_+$. With the help of the slicing (the thick line in Figure \ref{fig:51} denotes a path homologous to $\alpha_-$ in the slicing) we see that $\alpha_-$ can be transported to the path
	$$ [ 17,15,7,5,3,13,11,3,1,17,9,7,17 ] $$
	which entirely lies in $T_+$. This path is homologous to $-5$ times $\beta_+$ plus $4$ times $\alpha_+$ and hence the topological type of $C$ must be $L(-5,4) \cong L(5,1)$.

	\begin{figure}[p]
		\begin{center}
			\includegraphics[width=1.0\textwidth]{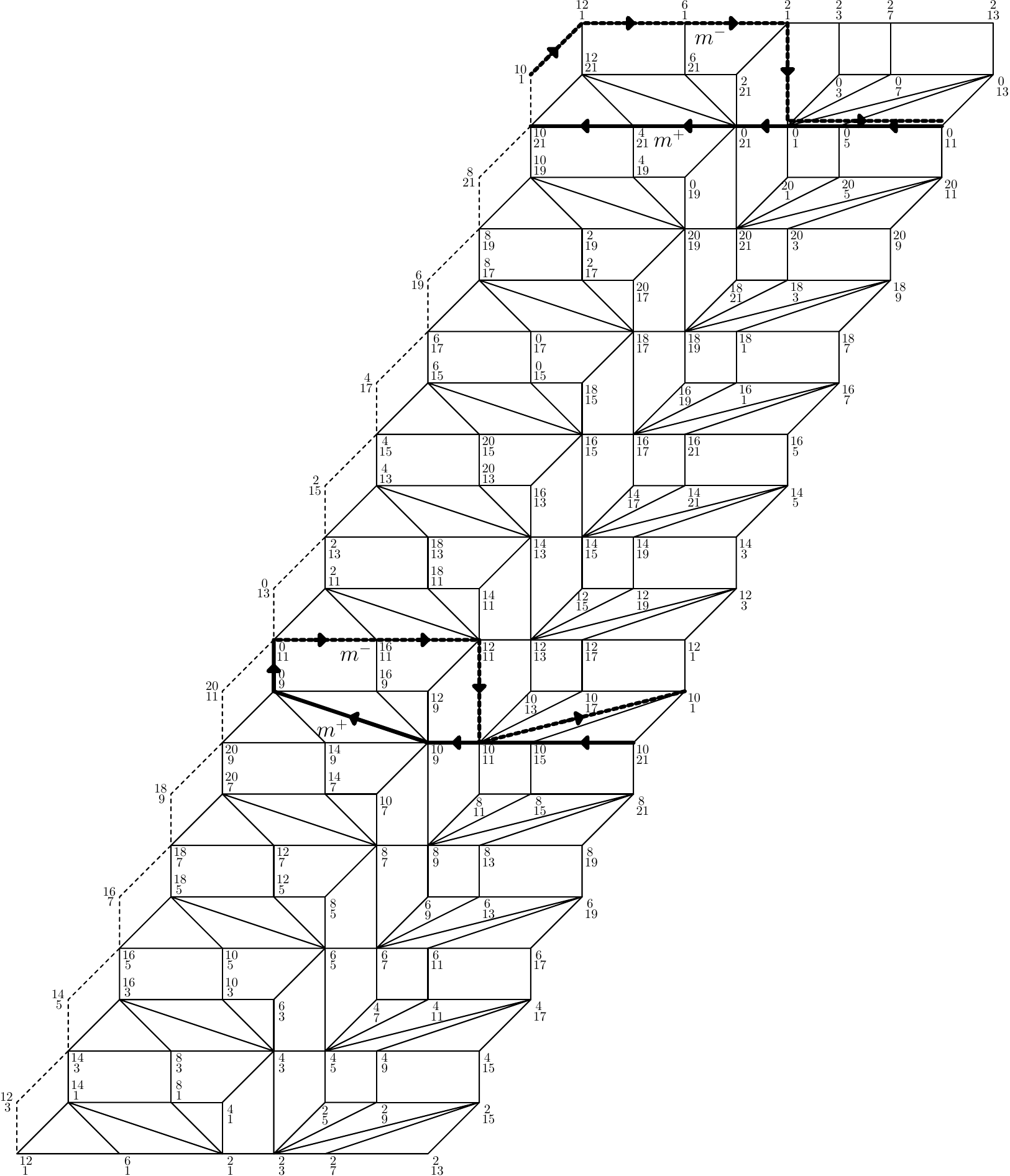}
			\caption{Slicing of $\operatorname{HS}$ between the odd and the even labeled vertices 
				splitting the complex into two isomorphic knot exteriors of the trefoil knot.\label{fig:homSphere}}
		\end{center}
	\end{figure}

	\medskip
	The proof that the cyclic combinatorial manifold $D$ with $22$ vertices defined by
		\scriptsize
		\begin{eqnarray}
			D&:=\{ &(1:1:1:19),(1:2:5:14),(1:7:12:2),(2:5:2:13),(2:7:2:11), \nonumber \\
			&&(2:8:4:8),(2:9:2:9),(2:12:3:5),(4:6:4:8),(4:6:6:6)\} \nonumber 
		\end{eqnarray}
		\normalsize
	is homeomorphic to the lens space $L(7,1)$ is completely analogous to the above.
	
%
	
	\medskip
	For the identification of the exact topological type of the lens spaces

		\scriptsize
		\begin{eqnarray}
			L_0 &:= \{&(1:1:1:11),(1:2:4:7),(1:4:2:7),(1:4:7:2),(2:4:4:4),(2:5:2:5)\} \nonumber \\
			L_1 &:= \{&(1:1:1:15),(1:2:4:11),(1:4:2:11),(1:4:11:2), \nonumber \\
				&&(2:4:8:4),(2:5:2:9),(2:7:2:7),(4:4:4:6)\} \nonumber \\
			L_2 &:= \{& (1:1:1:19),(1:2:4:15),(1:4:2:15),(1:4:15:2),(2:4:12:4), \nonumber \\
				&&(2:5:2:13),(2:7:2:11),(2:9:2:9),(4:4:4:10),(4:6:4:8)\} \nonumber
		\end{eqnarray}
		\normalsize

	\noindent
	with $14$, $18$ and $22$ vertices see Theorem \ref{thm:lensSeries}. Using the $3$-manifold software \texttt{regina} \cite{Burton09Regina}, we checked that all other lens spaces in the classification are of one of the above types.

	\medskip
	Finally, there are $74$ homeomorphic cyclic triangulations of a homology $3$-sphere with the lexicographically minimal complex being
	\scriptsize
	\begin{eqnarray}
		\operatorname{HS}&:=\{ &(1:1:1:19),(1:2:4:15),(1:4:8:9),(1:4:15:2),(1:6:6:9),  \nonumber \\
		&& (2:4:10:6),(2:5:6:9),(2:9:2:9),(2:9:5:6),(4:4:4:10)\} \nonumber
	\end{eqnarray}

	\normalsize
	(cf. Table \ref{fig:TopTypesOther}). In order to describe its topological type observe that the slicing between the even and the odd labeled vertices of $\operatorname{HS}$ 
	(cf. Figure \ref{fig:homSphere}) is a torus decomposing $\operatorname{HS}$ into two complexes $\operatorname{HS}^{+}$ (containing the even labeled vertices)
	and $\operatorname{HS}^{-}$ (containing the odd labeled vertices) which both are homeomorphic 
	to the knot exterior of the trefoil knot. The recognition step of this result is due to the $3$-manifold software \texttt{SnapPy} by Weeks \cite{WeeksSnapPea}
	which is well suited to deal with knot complements. 

\vspace{-.5cm}
\small
\begin{center}
	\begin{longtable}{|l|c|c|c|c|c|}
		\caption{The classification of cyclic combinatorial $3$-manifolds with up to $22$ vertices\label{fig:numbers}} \\
		
		\hline
		$n$ & \# complexes & \# cd$^*$ compl. & \# lm$^*$ compl. & \# cd lm$^*$ compl. & \# top. types \\ \hline 
		\endfirsthead

		\multicolumn{6}{l}%
		{ \tablename\ \thetable{} -- continued from previous page} \\
		\hline 
		$n$ & \# complexes & \# cd$^*$ compl. & \# lm$^*$ compl. & \# cd lm$^*$ compl. & \# top. types \\ \hline 
		\endhead

		\hline \multicolumn{6}{r}{{continued on next page --}} \\
		\endfoot

		\hline \hline
		\endlastfoot
		$5$&$1$&$1$&$1$&$1$&$1$ \\ \hline
		$6$&$1$&$1$&$0$&$0$&$1$ \\ \hline
		$7$&$3$&$1$&$0$&$0$&$1$ \\ \hline
		$8$&$3$&$2$&$0$&$0$&$1$ \\ \hline
		$9$&$6$&$2$&$3$&$1$&$2$ \\ \hline
		$10$&$19$&$8$&$0$&$0$&$3$ \\ \hline
		$11$&$40$&$6$&$0$&$0$&$2$ \\ \hline
		$12$&$56$&$20$&$0$&$0$&$4$ \\ \hline
		$13$&$135$&$15$&$0$&$0$&$2$ \\ \hline
		$14$&$258$&$50$&$0$&$0$&$4$ \\ \hline
		$15$&$217$&$34$&$1$&$1$&$5$ \\ \hline
		$16$&$742$&$107$&$12$&$2$&$8$ \\ \hline
		$17$&$1272$&$89$&$24$&$2$&$7$ \\ \hline
		$18$&$1818$&$319$&$24$&$4$&$15$ \\ \hline
		$19$&$4797$&$279$&$63$&$4$&$6$ \\ \hline
		$20$&$7670$&$1008$&$66$&$9$&$20$ \\ \hline
		$21$&$11931$&$1038$&$198$&$18$&$22$ \\ \hline
		$22$&$30550$&$3090$&$230$&$23$&$40$ \\ \hline
	\end{longtable}
\end{center}

\vspace{-.8cm}
$^*$ cd = combinatorially distinct, lm = locally minimal

\normalsize
\begin{landscape}
\tiny
\begin{center}
	\begin{longtable}{|r|c|c|c|l|l|}
		\caption{Cyclic combinatorial $3$-manifolds of spherical type. \label{fig:TopTypesSpherical}} \\
		
		\hline
		$n$ & top. type & $ \pi_1 $ & $\operatorname{TV}(7,1)^{***}$ & difference cycles of smallest complex* & source \\ \hline 
		\endfirsthead

		\multicolumn{6}{l}%
		{ \tablename\ \thetable{} -- continued from previous page} \\
		\hline 
		$n$ & top. type & $ \pi_1 $ & $\operatorname{TV}(7,1)^{***}$ & difference cycles of smallest complex* & source \\ \hline 
		\endhead

		\hline \multicolumn{6}{r}{{continued on next page --}} \\
		\endfoot
	
		\hline \hline
		\endlastfoot
$5$&$S^3$&$1$&$0.053787171163$&$\{(1\!:\!1\!:\!1\!:\!2)\}$&$\partial \Delta^4$ \\ \hline
$14$&$L(3,1)$&$\mathbb{Z}_3$&$0.174645847708$&$\{(1\!:\!1\!:\!1\!:\!11),(1\!:\!2\!:\!4\!:\!7),(1\!:\!4\!:\!2\!:\!7),(1\!:\!4\!:\!7\!:\!2),(2\!:\!4\!:\!4\!:\!4),(2\!:\!5\!:\!2\!:\!5)\}$&\cite[Complex $3_{14}$]{Kuehnel85NeighbComb3MfldsDihedralAutGroup}, Thm. \ref{thm:lensSeries} \\ \hline
$15$&$\mathbb{R}P^3$&$\mathbb{Z}_2$&&$\{(1\!:\!1\!:\!1\!:\!12),(1\!:\!2\!:\!3\!:\!9),(1\!:\!5\!:\!7\!:\!2),(2\!:\!3\!:\!3\!:\!7),(3\!:\!4\!:\!3\!:\!5),(3\!:\!4\!:\!4\!:\!4)\}$&\cite[Complex $2_{15}$]{Kuehnel85NeighbComb3MfldsDihedralAutGroup} \\ \hline
$15$&$P_2 = S^3 / Q_8$&$Q_8$&&$\{(1\!:\!1\!:\!1\!:\!12),(1\!:\!2\!:\!4\!:\!8),(1\!:\!6\!:\!6\!:\!2),(2\!:\!4\!:\!3\!:\!6),(3\!:\!4\!:\!4\!:\!4)\}$&\cite[Complex $8_{15}$]{Kuehnel85NeighbComb3MfldsDihedralAutGroup} \\ \hline
$16$&$S^3 / \operatorname{SL} (2,3)$&$\operatorname{SL}(2,3)$&&$\{(1\!:\!1\!:\!3\!:\!11),(1\!:\!1\!:\!4\!:\!10),(1\!:\!3\!:\!2\!:\!10),(2\!:\!3\!:\!8\!:\!3),(2\!:\!4\!:\!6\!:\!4),(3\!:\!5\!:\!3\!:\!5)\}$&\cite[Complex $^3 16^{1}_{31}$]{Lutz11TrigMnflds} \\ \hline
$17$&$\Sigma^3$&$\operatorname{SL}(2,5)$&$0.717779809966$&$\{(1\!:\!1\!:\!1\!:\!14),(1\!:\!2\!:\!4\!:\!10),(1\!:\!6\!:\!8\!:\!2),(2\!:\!3\!:\!4\!:\!8),(2\!:\!3\!:\!6\!:\!6),(2\!:\!4\!:\!5\!:\!6),(4\!:\!4\!:\!4\!:\!5)\}$&\cite[Complex $^3 17^{1}_{21}$]{Lutz11TrigMnflds} \\ \hline
$18$&$L(8,3)$&$\mathbb{Z}_8$&&$\{(1\!:\!1\!:\!1\!:\!15),(1\!:\!2\!:\!4\!:\!11),(1\!:\!4\!:\!2\!:\!11),(1\!:\!4\!:\!11\!:\!2),(2\!:\!4\!:\!8\!:\!4),(2\!:\!5\!:\!2\!:\!9),(2\!:\!7\!:\!2\!:\!7),(4\!:\!4\!:\!4\!:\!6)\}$&Thm. \ref{thm:lensSeries} \\ \hline
$18$&$L(5,1)$&$\mathbb{Z}_5$&&$\{(1\!:\!1\!:\!1\!:\!15),(1\!:\!2\!:\!5\!:\!10),(1\!:\!5\!:\!2\!:\!10),(1\!:\!5\!:\!10\!:\!2),(2\!:\!5\!:\!2\!:\!9),(2\!:\!6\!:\!4\!:\!6),(2\!:\!7\!:\!2\!:\!7),(4\!:\!4\!:\!4\!:\!6)\}$& \\ \hline
$20$&$P_7 = S^3/Q_{28}$&$\mathbb{Z}_7 \ltimes \mathbb{Z}_4$&&$\{(1\!:\!1\!:\!1\!:\!17),(1\!:\!2\!:\!15\!:\!2),(2\!:\!3\!:\!12\!:\!3),(3\!:\!4\!:\!6\!:\!7),(3\!:\!4\!:\!7\!:\!6),(3\!:\!5\!:\!3\!:\!9),(3\!:\!6\!:\!3\!:\!8),(3\!:\!6\!:\!4\!:\!7),(4\!:\!6\!:\!4\!:\!6)\}$& \\ \hline
$22$&$P_8 = S^3/Q_{32}$&$Q_{32}$&&$\{(1\!:\!1\!:\!1\!:\!19),(1\!:\!2\!:\!1\!:\!18),(1\!:\!3\!:\!15\!:\!3),(3\!:\!4\!:\!3\!:\!12),(3\!:\!5\!:\!6\!:\!8),(3\!:\!5\!:\!8\!:\!6),(3\!:\!6\!:\!3\!:\!10),(3\!:\!6\!:\!5\!:\!8),(3\!:\!7\!:\!3\!:\!9),(5\!:\!6\!:\!5\!:\!6)\}$& \\ \hline
$22$&$P_4 = S^3/Q_{16}$&$Q_{16}$&&$\{(1\!:\!1\!:\!1\!:\!19),(1\!:\!2\!:\!3\!:\!16),(1\!:\!5\!:\!7\!:\!9),(1\!:\!12\!:\!7\!:\!2),(2\!:\!3\!:\!3\!:\!14),(2\!:\!6\!:\!7\!:\!7),(3\!:\!4\!:\!8\!:\!7),(3\!:\!4\!:\!10\!:\!5),(4\!:\!4\!:\!4\!:\!10)\}$& \\ \hline
$22$&$L(15,4)$&$\mathbb{Z}_{15}$&&$\{(1\!:\!1\!:\!1\!:\!19),(1\!:\!2\!:\!4\!:\!15),(1\!:\!4\!:\!2\!:\!15),(1\!:\!4\!:\!15\!:\!2),(2\!:\!4\!:\!12\!:\!4),(2\!:\!5\!:\!2\!:\!13),(2\!:\!7\!:\!2\!:\!11),(2\!:\!9\!:\!2\!:\!9),(4\!:\!4\!:\!4\!:\!10),(4\!:\!6\!:\!4\!:\!8)\}$&Thm. \ref{thm:lensSeries} \\ \hline
$22$&$L(7,1)$&$\mathbb{Z}_7$&&$\{(1\!:\!1\!:\!1\!:\!19),(1\!:\!2\!:\!5\!:\!14),(1\!:\!7\!:\!12\!:\!2),(2\!:\!5\!:\!2\!:\!13),(2\!:\!7\!:\!2\!:\!11),(2\!:\!8\!:\!4\!:\!8),(2\!:\!9\!:\!2\!:\!9),(2\!:\!12\!:\!3\!:\!5),(4\!:\!6\!:\!4\!:\!8),(4\!:\!6\!:\!6\!:\!6)\}$& \\ \hline
	\end{longtable}
\end{center}

\begin{center}
	\begin{longtable}{|r|c|c|c|l|l|}
		\caption{Cyclic combinatorial $3$-manifolds of type $S^2 \times \mathbb{R}$. \label{fig:TopTypesS2xR}} \\
		
		\hline
		$n$ & top. type & $ H_{*} $ & $\operatorname{TV}(7,1)^{***}$ & difference cycles of smallest complex* & source \\ \hline 
		\endfirsthead

		\multicolumn{5}{l}%
		{ \tablename\ \thetable{} -- continued from previous page} \\
		\hline 
		$n$ & top. type & $ H_{*} $ & $\operatorname{TV}(7,1)^{***}$ & difference cycles of smallest complex* & source \\ \hline 
		\endhead

		\hline \multicolumn{5}{r}{{continued on next page --}} \\
		\endfoot
		
		\hline \hline
		\endlastfoot
$9$&$S^2 \dtimes S^1$&$(\mathbb{Z},\mathbb{Z} ,\mathbb{Z}_{2},0)$&&$\{(1\!:\!1\!:\!2\!:\!5),(1\!:\!1\!:\!5\!:\!2),(1\!:\!2\!:\!1\!:\!5)\}$&\cite[Complex $N^9_{51}$]{Altshuler74NeighComb3Mnf9Vert} \\ \hline
$10$&$S^2 \times S^1$&$(\mathbb{Z},\mathbb{Z} ,\mathbb{Z} ,\mathbb{Z})$&$1$&$\{(1\!:\!1\!:\!2\!:\!6),(1\!:\!1\!:\!6\!:\!2),(1\!:\!2\!:\!1\!:\!6)\}$&\cite[Complex $M^3_2(10)$]{Kuehnel96PermDiffCyc} \\ \hline
$17$&$\mathbb{R}P^2 \times S^1$&$(\mathbb{Z},\mathbb{Z} \oplus \mathbb{Z}_{2},\mathbb{Z}_{2},0)$&&$\{(1\!:\!1\!:\!1\!:\!14),(1\!:\!2\!:\!12\!:\!2),(2\!:\!3\!:\!9\!:\!3),(
3\!:\!4\!:\!4\!:\!6),(3\!:\!4\!:\!6\!:\!4),(3\!:\!5\!:\!3\!:\!6),(3\!:\!6\!:\!4\!:\!4)\}$&\cite[Complex $IV_{17}$]{Kuehnel85NeighbComb3MfldsDihedralAutGroup}, \cite[Complex $^3 17^{2}_{13}$]{Lutz11TrigMnflds} \\ \hline
	\end{longtable}
\end{center}

\begin{center}
	\begin{longtable}{|r|c|c|c|l|l|}
		\caption{Cyclic combinatorial $3$-manifolds of flat type. \label{fig:TopTypesFlat}} \\
		
		\hline
		$n$ & top. type & $ H_{*} $ & $\operatorname{TV}(7,1)^{***}$ & difference cycles of smallest complex* & source \\ \hline 
		\endfirsthead

		\multicolumn{5}{l}%
		{ \tablename\ \thetable{} -- continued from previous page} \\
		\hline 
		$n$ & top. type & $ H_{*} $ & $\operatorname{TV}(7,1)^{***}$ & difference cycles of smallest complex* & source \\ \hline 
		\endhead

		\hline \multicolumn{5}{r}{{continued on next page --}} \\
		\endfoot
		
		\hline \hline
		\endlastfoot
$15$&$\mathbb{T}^3$&$(\mathbb{Z},\mathbb{Z}^{3} ,\mathbb{Z}^{3} ,\mathbb{Z})$&&$\{(1\!:\!2\!:\!4\!:\!8),(1\!:\!2\!:\!8\!:\!4),(1\!:\!4\!:\!2\!:\!8),(1\!:\!4\!:\!
8\!:\!2),(1\!:\!8\!:\!2\!:\!4),(1\!:\!8\!:\!4\!:\!2)\}$&\cite[Complex $III_{15}$]{Kuehnel85NeighbComb3MfldsDihedralAutGroup} \\ \hline
$16$&$\mathfrak{B}_2$&$(\mathbb{Z},\mathbb{Z}^{2} ,\mathbb{Z} \oplus \mathbb{Z}_{2},0)$&$6$&$\{(1\!:\!1\!:\!3\!:\!11),(1\!:\!1\!:\!4\!:\!10),(1\!:\!3\!:\!2\!:\!10),(2\!:\!
3\!:\!4\!:\!7),(2\!:\!4\!:\!7\!:\!3),(2\!:\!7\!:\!3\!:\!4)\}$&\cite[p.\ 89]{Lutz03TrigMnfFewVertVertTrans}, \cite[Complex $^3 16^{55}_{10}$]{Lutz11TrigMnflds} \\ \hline
$18$&$\mathbb{K}^2 \times S^1$&$(\mathbb{Z},\mathbb{Z}^{2} \oplus \mathbb{Z}_{2},\mathbb{Z} \oplus \mathbb{Z}_{2},0)$&&$\{(1\!:\!1\!:\!3\!:\!13),(1\!:\!1\!:\!6\!:\!10),(
1\!:\!3\!:\!1\!:\!13),(1\!:\!6\!:\!8\!:\!3),(1\!:\!7\!:\!6\!:\!4),(2\!:\!3\!:\!7\!:\!6),(2\!:\!6\!:\!4\!:\!6)\}$& \\ \hline
$18$&$\mathfrak{B}_4$&$(\mathbb{Z},\mathbb{Z} \oplus \mathbb{Z}_{4},\mathbb{Z}_{2},0)$&&$\{(1\!:\!1\!:\!3\!:\!13),(1\!:\!1\!:\!13\!:\!3),(1\!:\!3\!:\!1\!:\!13),(2\!:\!
3\!:\!6\!:\!7),(2\!:\!6\!:\!2\!:\!8),(2\!:\!6\!:\!7\!:\!3),(2\!:\!7\!:\!2\!:\!7),(2\!:\!7\!:\!3\!:\!6)\}$& \\ \hline
$20$&$\mathfrak{G}_2$&$(\mathbb{Z},\mathbb{Z} \oplus \mathbb{Z}_{2} \oplus \mathbb{Z}_{2},\mathbb{Z} ,\mathbb{Z})$&&$\{(1\!:\!1\!:\!3\!:\!15),(1\!:\!1\!:\!6\!:\!12),(
1\!:\!3\!:\!1\!:\!15),(1\!:\!6\!:\!10\!:\!3),(1\!:\!7\!:\!2\!:\!10),(1\!:\!9\!:\!6\!:\!4),(2\!:\!3\!:\!7\!:\!8),(2\!:\!6\!:\!5\!:\!7),(4\!:\!6\!:\!4\!:\!
6)\}$& \\ \hline
$21$&$\mathfrak{G}_3$&$(\mathbb{Z},\mathbb{Z} \oplus \mathbb{Z}_{3},\mathbb{Z} ,\mathbb{Z})$&&$\{(1\!:\!1\!:\!1\!:\!18),(1\!:\!2\!:\!1\!:\!17),(1\!:\!3\!:\!5\!:\!12),(
1\!:\!5\!:\!3\!:\!12),(1\!:\!5\!:\!6\!:\!9),(1\!:\!11\!:\!6\!:\!3),(3\!:\!5\!:\!8\!:\!5),(4\!:\!5\!:\!6\!:\!6)\}$& \\ \hline
	\end{longtable}
\end{center}

\begin{center}
	\begin{longtable}{|r|c|c|l|}
		\caption{Cyclic combinatorial $3$-manifolds of $Nil$ type. \label{fig:TopTypesNil}} \\
		
		\hline
		$n$ & top. type & $ H_{*} $ & difference cycles of smallest complex* \\ \hline 
		\endfirsthead

		\multicolumn{4}{l}%
		{ \tablename\ \thetable{} -- continued from previous page} \\
		\hline 
		$n$ & top. type & $ H_{*} $ & difference cycles of smallest complex* \\ \hline 
		\endhead

		\hline \multicolumn{4}{r}{{continued on next page --}} \\
		\endfoot
		
		\hline \hline
		\endlastfoot
$18$&$\operatorname{SFS} [\mathbb{T}^2: (1,1)]$&$(\mathbb{Z},\mathbb{Z}^{2} ,\mathbb{Z}^{2} ,\mathbb{Z})$&$\{(1\!:\!2\!:\!2\!:\!13),(1\!:\!2\!:\!12\!:\!3),(1\!:\!4\!:\!2\!:\!11),(
1\!:\!6\!:\!1\!:\!10),(1\!:\!7\!:\!1\!:\!9),(1\!:\!8\!:\!1\!:\!8),(1\!:\!11\!:\!3\!:\!3),(2\!:\!2\!:\!11\!:\!3)\}$ \\ \hline
$20$&$\operatorname{SFS} [\mathbb{K}^2/n2: (1,5)]$&$(\mathbb{Z},\mathbb{Z} \oplus \mathbb{Z}_{4},\mathbb{Z} ,\mathbb{Z})$&$\{(1\!:\!1\!:\!1\!:\!17),(1\!:\!2\!:\!4\!:\!13),(1\!:\!
6\!:\!8\!:\!5),(1\!:\!8\!:\!6\!:\!5),(1\!:\!8\!:\!9\!:\!2),(2\!:\!4\!:\!5\!:\!9),(3\!:\!4\!:\!4\!:\!9),(4\!:\!4\!:\!5\!:\!7)\}$ \\ \hline
$21$&$\operatorname{SFS} [S^2: (3,2) (3,2) (3,-1)]$&$(\mathbb{Z},\mathbb{Z}_{3} \oplus \mathbb{Z}_{9},0,\mathbb{Z})$&$\{(1\!:\!1\!:\!1\!:\!18),(1\!:\!2\!:\!1\!:\!17),(
1\!:\!3\!:\!6\!:\!11),(1\!:\!6\!:\!3\!:\!11),(1\!:\!6\!:\!11\!:\!3),(3\!:\!5\!:\!3\!:\!10),(3\!:\!5\!:\!8\!:\!5),(3\!:\!6\!:\!6\!:\!6),(3\!:\!7\!:\!3\!:\!
8)\}$ \\ \hline
$21$&$\operatorname{SFS} [\mathbb{T}^2: (1,7)]$&$(\mathbb{Z},\mathbb{Z}^{2} \oplus \mathbb{Z}_{7},\mathbb{Z}^{2} ,\mathbb{Z})$&$\{(1\!:\!2\!:\!4\!:\!14),(1\!:\!2\!:\!5\!:\!13),(
1\!:\!6\!:\!5\!:\!9),(1\!:\!7\!:\!9\!:\!4),(1\!:\!11\!:\!3\!:\!6),(1\!:\!14\!:\!2\!:\!4),(3\!:\!4\!:\!3\!:\!11),(3\!:\!5\!:\!9\!:\!4),(3\!:\!6\!:\!5\!:\!
7)\}$ \\ \hline
	\end{longtable}
\end{center}

\begin{center}
	\begin{longtable}{|r|c|c|c|l|}
		\caption{Cyclic combinatorial $3$-manifolds of type $SL(2,\mathbb{R})$. \label{fig:TopTypesSL2R}} \\
		
		\hline
		$n$ & top. type & $ H_{*} $ & $\operatorname{TV}(7,1)^{***}$ & difference cycles of smallest complex* \\ \hline 
		\endfirsthead

		\multicolumn{4}{l}%
		{ \tablename\ \thetable{} -- continued from previous page} \\
		\hline 
		$n$ & top. type & $ H_{*} $ & $\operatorname{TV}(7,1)^{***}$ & difference cycles of smallest complex* \\ \hline 
		\endhead

		\hline \multicolumn{4}{r}{{continued on next page --}} \\
		\endfoot
		
		\hline \hline
		\endlastfoot
$19$&$\Sigma(2,3,7)$&$(\mathbb{Z},0,0,\mathbb{Z})$&$0.881772448769$&$\{(1\!:\!1\!:\!1\!:\!16),(1\!:\!2\!:\!6\!:\!10),(1\!:\!8\!:\!8\!:\!2),(2\!:\!6\!:\!3\!:\!8),(3\!:\!
6\!:\!4\!:\!6),(4\!:\!5\!:\!4\!:\!6),(4\!:\!5\!:\!5\!:\!5)\}$ \\ \hline
$20$&$\operatorname{SFS} [S^2: (3,1) (3,1) (4,-3)]$&$(\mathbb{Z},\mathbb{Z}_{3},0,\mathbb{Z})$&$0.838638486511$&$\{(1\!:\!1\!:\!5\!:\!13),(1\!:\!1\!:\!6\!:\!12),(1\!:\!5\!:\!2\!:\!12),(
2\!:\!5\!:\!2\!:\!11),(2\!:\!6\!:\!6\!:\!6),(2\!:\!7\!:\!4\!:\!7),(3\!:\!4\!:\!3\!:\!10),(3\!:\!4\!:\!9\!:\!4),(3\!:\!7\!:\!3\!:\!7)\}$ \\ \hline
$21$&$\operatorname{SFS} [S^2: (2,1) (2,1) (2,1) (3,-5)]$&$(\mathbb{Z},\mathbb{Z}_{2} \oplus \mathbb{Z}_{2},0,\mathbb{Z})$&&$\{(1\!:\!1\!:\!1\!:\!18),(1\!:\!2\!:\!7\!:\!
11),(1\!:\!9\!:\!9\!:\!2),(2\!:\!7\!:\!3\!:\!9),(3\!:\!7\!:\!4\!:\!7),(4\!:\!6\!:\!4\!:\!7),(4\!:\!6\!:\!5\!:\!6),(5\!:\!5\!:\!5\!:\!6)\}$ \\ \hline
$21$&$\operatorname{SFS} [S^2: (5,1) (5,1) (5,-4)]$&$(\mathbb{Z},\mathbb{Z}_{5} \oplus \mathbb{Z}_{10},0,\mathbb{Z})$&&$\{(1\!:\!1\!:\!3\!:\!16),(1\!:\!1\!:\!4\!:\!15),(
1\!:\!3\!:\!10\!:\!7),(1\!:\!5\!:\!8\!:\!7),(2\!:\!3\!:\!6\!:\!10),(2\!:\!4\!:\!10\!:\!5),(2\!:\!6\!:\!3\!:\!10),(2\!:\!6\!:\!8\!:\!5),(3\!:\!6\!:\!3\!:\!
9)\}$ \\ \hline
$21$&$\operatorname{SFS} [S^2: (4,1) (4,1) (4,-3)]$&$(\mathbb{Z},\mathbb{Z}_{4} \oplus \mathbb{Z}_{4},0,\mathbb{Z})$&&$\{(1\!:\!1\!:\!3\!:\!16),(1\!:\!1\!:\!4\!:\!15),(
1\!:\!3\!:\!10\!:\!7),(1\!:\!5\!:\!8\!:\!7),(2\!:\!3\!:\!6\!:\!10),(2\!:\!4\!:\!10\!:\!5),(2\!:\!6\!:\!3\!:\!10),(2\!:\!6\!:\!8\!:\!5),(3\!:\!6\!:\!6\!:\!
6)\}$ \\ \hline
$22$&$\operatorname{SFS} [S^2: (4,1) (5,2) (5,-3)]$&$(\mathbb{Z},\mathbb{Z}_{5},0,\mathbb{Z})$&&$\{(1\!:\!1\!:\!1\!:\!19),(1\!:\!2\!:\!17\!:\!2),(2\!:\!3\!:\!4\!:\!13),(
2\!:\!7\!:\!10\!:\!3),(3\!:\!4\!:\!8\!:\!7),(3\!:\!5\!:\!4\!:\!10),(3\!:\!5\!:\!7\!:\!7),(4\!:\!6\!:\!4\!:\!8),(4\!:\!6\!:\!6\!:\!6)\}$ \\ \hline
$22$&$\operatorname{SFS} [S^2: (3,1) (3,1) (9,-7)]$&$(\mathbb{Z},\mathbb{Z}_{3} \oplus \mathbb{Z}_{3},0,\mathbb{Z})$&$0.107574342326$&$\{(1\!:\!1\!:\!1\!:\!19),(1\!:\!2\!:\!17\!:\!2),(
2\!:\!3\!:\!14\!:\!3),(3\!:\!5\!:\!9\!:\!5),(4\!:\!5\!:\!6\!:\!7),(4\!:\!5\!:\!8\!:\!5),(4\!:\!6\!:\!6\!:\!6),(4\!:\!6\!:\!7\!:\!5),(4\!:\!7\!:\!4\!:\!7),(
4\!:\!7\!:\!5\!:\!6)\}$ \\ \hline
$22$&$\operatorname{SFS} [S^2: (2,1) (2,1) (3,1) (3,-2)]$&$(\mathbb{Z},\mathbb{Z}_{24},0,\mathbb{Z})$&&$\{(1\!:\!1\!:\!3\!:\!17),(1\!:\!1\!:\!4\!:\!16),(1\!:\!3\!:\!2\!:\!
16),(2\!:\!3\!:\!7\!:\!10),(2\!:\!4\!:\!2\!:\!14),(2\!:\!6\!:\!8\!:\!6),(2\!:\!7\!:\!2\!:\!11),(2\!:\!7\!:\!10\!:\!3),(2\!:\!9\!:\!2\!:\!9),(2\!:\!10\!:\!
3\!:\!7)\}$ \\ \hline
$22$&$\operatorname{SFS} [S^2: (3,2) (4,1) (4,-3)]$&$(\mathbb{Z},\mathbb{Z}_{8},0,\mathbb{Z})$&&$\{(1\!:\!1\!:\!3\!:\!17),(1\!:\!1\!:\!4\!:\!16),(1\!:\!3\!:\!2\!:\!16),(
2\!:\!3\!:\!14\!:\!3),(2\!:\!4\!:\!12\!:\!4),(3\!:\!5\!:\!3\!:\!11),(3\!:\!8\!:\!3\!:\!8),(4\!:\!6\!:\!6\!:\!6)\}$ \\ \hline
$22$&$\operatorname{SFS} [S^2: (3,1) (3,1) (5,-3)]$&$(\mathbb{Z},\mathbb{Z}_{3},0,\mathbb{Z})$&$0.543133962258$&$\{(1\!:\!1\!:\!5\!:\!15),(1\!:\!1\!:\!6\!:\!14),(1\!:\!5\!:\!2\!:\!14),(
2\!:\!5\!:\!10\!:\!5),(2\!:\!6\!:\!8\!:\!6),(3\!:\!4\!:\!3\!:\!12),(3\!:\!4\!:\!11\!:\!4),(3\!:\!7\!:\!5\!:\!7),(4\!:\!7\!:\!4\!:\!7)\}$ \\ \hline
	\end{longtable}
\end{center}

\begin{center}
	\begin{longtable}{|r|c|c|c|l|}
		\caption{Cyclic combinatorial $3$-manifolds of type $\mathbb{H}^2 \times \mathbb{R}$. \label{fig:TopTypesH2xR}} \\
		
		\hline
		$n$ & top. type & $ H_{*} $ & $\operatorname{TV}(7,1)^{***}$ & difference cycles of smallest complex* \\ \hline 
		\endfirsthead

		\multicolumn{4}{l}%
		{ \tablename\ \thetable{} -- continued from previous page} \\
		\hline 
		$n$ & top. type & $ H_{*} $ & $\operatorname{TV}(7,1)^{***}$ & difference cycles of smallest complex* \\ \hline 
		\endhead

		\hline \multicolumn{4}{r}{{continued on next page --}} \\
		\endfoot
		
		\hline \hline
		\endlastfoot
$18$&$\operatorname{SFS}[\mathbb{R}P^2:(2,1)(2,1)(2,1)]$&$(\mathbb{Z},\mathbb{Z} \oplus \mathbb{Z}_{2} \oplus \mathbb{Z}_{2},\mathbb{Z}_{2},0)$&&$\{(1\!:\!1\!:\!1\!:\!15),(1\!:\!
2\!:\!5\!:\!10),(1\!:\!4\!:\!3\!:\!10),(1\!:\!4\!:\!11\!:\!2),(3\!:\!4\!:\!5\!:\!6),(3\!:\!5\!:\!6\!:\!4),(3\!:\!6\!:\!3\!:\!6),(3\!:\!6\!:\!4\!:\!
5)\}$ \\ \hline
$19$&$\operatorname{SFS}[\mathbb{R}P^2:(2,1)(3,1)]$&$(\mathbb{Z},\mathbb{Z} ,\mathbb{Z}_{2},0)$&&$\{(1\!:\!1\!:\!1\!:\!16),(1\!:\!2\!:\!5\!:\!11),(1\!:\!4\!:\!3\!:\!11),(1\!:\!
4\!:\!12\!:\!2),(3\!:\!4\!:\!6\!:\!6),(3\!:\!5\!:\!3\!:\!8),(3\!:\!6\!:\!4\!:\!6),(3\!:\!6\!:\!6\!:\!4)\}$ \\ \hline
$20$&$\operatorname{SFS}[\mathbb{R}P^2:(3,1)(3,2)]$&$(\mathbb{Z},\mathbb{Z} \oplus \mathbb{Z}_{3},\mathbb{Z}_{2},0)$&&$\{(1\!:\!1\!:\!1\!:\!17),(1\!:\!2\!:\!5\!:\!12),(1\!:\!
5\!:\!2\!:\!12),(1\!:\!5\!:\!12\!:\!2),(2\!:\!5\!:\!4\!:\!9),(2\!:\!6\!:\!6\!:\!6),(2\!:\!9\!:\!4\!:\!5),(4\!:\!5\!:\!4\!:\!7)\}$ \\ \hline
$21$&$(\mathbb{R}P^2)^{\#3} \times S^1$&$(\mathbb{Z},\mathbb{Z}^{3} \oplus \mathbb{Z}_{2},\mathbb{Z}^{2} \oplus \mathbb{Z}_{2},0)$&&$\{(1\!:\!1\!:\!1\!:\!18),(1\!:\!2\!:\!
1\!:\!17),(1\!:\!3\!:\!6\!:\!11),(1\!:\!6\!:\!3\!:\!11),(1\!:\!6\!:\!11\!:\!3),(3\!:\!5\!:\!6\!:\!7),(3\!:\!5\!:\!8\!:\!5),(3\!:\!6\!:\!7\!:\!5),(3\!:\!
7\!:\!5\!:\!6)\}$ \\ \hline
$21$&$\operatorname{SFS}[\mathbb{R}P^2:(3,1)(3,1)(3,2)]$&$(\mathbb{Z},\mathbb{Z} \oplus \mathbb{Z}_{3} \oplus \mathbb{Z}_{6},\mathbb{Z}_{2},0)$&&$\{(1\!:\!1\!:\!1\!:\!18),(1\!:\!
2\!:\!5\!:\!13),(1\!:\!4\!:\!3\!:\!13),(1\!:\!4\!:\!14\!:\!2),(3\!:\!4\!:\!3\!:\!11),(3\!:\!5\!:\!6\!:\!7),(3\!:\!6\!:\!6\!:\!6),(3\!:\!6\!:\!7\!:\!5),(
3\!:\!7\!:\!5\!:\!6)\}$ \\ \hline
$21$&$\operatorname{SFS}[\mathbb{R}P^2:(3,1)(3,1)(3,1)]$&$(\mathbb{Z},\mathbb{Z} \oplus \mathbb{Z}_{3} \oplus \mathbb{Z}_{3},\mathbb{Z}_{2},0)$&&$\{(1\!:\!1\!:\!1\!:\!18),(1\!:\!
2\!:\!6\!:\!12),(1\!:\!4\!:\!4\!:\!12),(1\!:\!4\!:\!14\!:\!2),(2\!:\!5\!:\!4\!:\!10),(2\!:\!6\!:\!3\!:\!10),(3\!:\!4\!:\!10\!:\!4),(3\!:\!6\!:\!6\!:\!6),(
3\!:\!10\!:\!4\!:\!4)\}$ \\ \hline
$21$&$\operatorname{SFS}[\mathbb{K}^2:(2,1)]$&$(\mathbb{Z},\mathbb{Z}^{2} ,\mathbb{Z} \oplus \mathbb{Z}_{2},0)$&$12$&$\{(1\!:\!1\!:\!3\!:\!16),(1\!:\!1\!:\!13\!:\!6),(1\!:\!3\!:\!
11\!:\!6),(1\!:\!4\!:\!9\!:\!7),(2\!:\!3\!:\!6\!:\!10),(2\!:\!6\!:\!4\!:\!9),(2\!:\!6\!:\!10\!:\!3),(2\!:\!9\!:\!7\!:\!3),(2\!:\!10\!:\!3\!:\!
6)\}$ \\ \hline
$21$&$\operatorname{SFS}[\mathbb{K}^2:(2,1)(2,1)(2,1)]$&$(\mathbb{Z},\mathbb{Z}^{2} \oplus \mathbb{Z}_{2} \oplus \mathbb{Z}_{2},\mathbb{Z} \oplus \mathbb{Z}_{2},0)$&&$\{(1\!:\!
2\!:\!3\!:\!15),(1\!:\!2\!:\!13\!:\!5),(1\!:\!5\!:\!2\!:\!13),(1\!:\!7\!:\!4\!:\!9),(1\!:\!11\!:\!4\!:\!5),(2\!:\!3\!:\!3\!:\!13),(2\!:\!6\!:\!4\!:\!9),(
2\!:\!10\!:\!4\!:\!5),(4\!:\!6\!:\!4\!:\!7)\}$ \\ \hline
$22$&$\operatorname{SFS}[\mathbb{R}P^2:(3,1)(4,3)]$&$(\mathbb{Z},\mathbb{Z} ,\mathbb{Z}_{2},0)$&&$\{(1\!:\!1\!:\!1\!:\!19),(1\!:\!2\!:\!4\!:\!15),(1\!:\!6\!:\!13\!:\!2),(2\!:\!
4\!:\!3\!:\!13),(3\!:\!4\!:\!11\!:\!4),(3\!:\!5\!:\!5\!:\!9),(3\!:\!5\!:\!9\!:\!5),(3\!:\!6\!:\!3\!:\!10),(3\!:\!9\!:\!5\!:\!5),(4\!:\!7\!:\!4\!:\!
7)\}$ \\ \hline
$22$&$\operatorname{SFS}[D_:(3,1)(3,1)]$&$(\mathbb{Z},\mathbb{Z} \oplus \mathbb{Z}_{3},\mathbb{Z}_{2},0)$&&$\{(1\!:\!1\!:\!1\!:\!19),(1\!:\!2\!:\!10\!:\!9),(1\!:\!5\!:\!
7\!:\!9),(1\!:\!5\!:\!14\!:\!2),(2\!:\!6\!:\!8\!:\!6),(2\!:\!10\!:\!4\!:\!6),(3\!:\!5\!:\!10\!:\!4),(3\!:\!10\!:\!5\!:\!4),(4\!:\!6\!:\!7\!:\!
5)\}$ \\ \hline
$22$&$\operatorname{SFS}[\mathbb{R}P^2:(2,1)(5,1)]$&$(\mathbb{Z},\mathbb{Z} ,\mathbb{Z}_{2},0)$&&$\{(1\!:\!1\!:\!1\!:\!19),(1\!:\!2\!:\!17\!:\!2),(2\!:\!3\!:\!4\!:\!13),(2\!:\!
7\!:\!10\!:\!3),(3\!:\!4\!:\!5\!:\!10),(4\!:\!5\!:\!6\!:\!7),(4\!:\!6\!:\!5\!:\!7),(4\!:\!6\!:\!7\!:\!5),(5\!:\!6\!:\!5\!:\!6)\}$ \\ \hline
$22$&$\operatorname{SFS}[\mathbb{K}^2:(3,1)]$&$(\mathbb{Z},\mathbb{Z}^{2} ,\mathbb{Z} \oplus \mathbb{Z}_{2},0)$&$10$&$\{(1\!:\!1\!:\!3\!:\!17),(1\!:\!1\!:\!4\!:\!16),(1\!:\!3\!:\!
2\!:\!16),(2\!:\!3\!:\!14\!:\!3),(2\!:\!4\!:\!9\!:\!7),(2\!:\!7\!:\!6\!:\!7),(2\!:\!7\!:\!9\!:\!4),(3\!:\!5\!:\!3\!:\!11),(3\!:\!8\!:\!3\!:\!
8)\}$ \\ \hline
$22$&$\operatorname{SFS}[S^2:(2,1)(2,1)(3,1)(3,-4)]$&$(\mathbb{Z},\mathbb{Z} ,\mathbb{Z} ,\mathbb{Z})$&$4$&$\{(1\!:\!1\!:\!3\!:\!17),(1\!:\!1\!:\!6\!:\!14),(1\!:\!3\!:\!
1\!:\!17),(1\!:\!6\!:\!2\!:\!13),(1\!:\!7\!:\!10\!:\!4),(1\!:\!8\!:\!10\!:\!3),(2\!:\!3\!:\!7\!:\!10),(2\!:\!6\!:\!8\!:\!6),(2\!:\!10\!:\!3\!:\!
7)\}$ \\ \hline
$22$&$\operatorname{SFS}[\mathbb{R}P^2:(5,2)(5,3)]$&$(\mathbb{Z},\mathbb{Z} \oplus \mathbb{Z}_{5},\mathbb{Z}_{2},0)$&&$\{(1\!:\!1\!:\!5\!:\!15),(1\!:\!1\!:\!15\!:\!5),(1\!:\!
5\!:\!1\!:\!15),(2\!:\!3\!:\!2\!:\!15),(2\!:\!3\!:\!8\!:\!9),(2\!:\!8\!:\!4\!:\!8),(2\!:\!8\!:\!9\!:\!3),(2\!:\!9\!:\!2\!:\!9),(2\!:\!9\!:\!3\!:\!8),(
4\!:\!4\!:\!4\!:\!10)\}$ \\ \hline
$22$&$\operatorname{SFS}[(\mathbb{R}P^2)^{\#3} :(1,1)]$&$(\mathbb{Z},\mathbb{Z}^{3} ,\mathbb{Z}^{2} \oplus \mathbb{Z}_{2},0)$&&$\{(1\!:\!1\!:\!5\!:\!15),(1\!:\!1\!:\!15\!:\!5),(
1\!:\!5\!:\!1\!:\!15),(2\!:\!5\!:\!3\!:\!12),(2\!:\!8\!:\!4\!:\!8),(2\!:\!9\!:\!2\!:\!9),(2\!:\!9\!:\!3\!:\!8),(2\!:\!11\!:\!4\!:\!5),(3\!:\!8\!:\!4\!:\!
7),(3\!:\!10\!:\!4\!:\!5)\}$ \\ \hline
$22$&$\operatorname{SFS}[\mathbb{K}^2:(2,1)(2,1)]$&$(\mathbb{Z},\mathbb{Z}^{2} \oplus \mathbb{Z}_{4},\mathbb{Z} \oplus \mathbb{Z}_{2},0)$&&$\{(1\!:\!2\!:\!4\!:\!15),(1\!:\!2\!:\!
13\!:\!6),(1\!:\!4\!:\!2\!:\!15),(1\!:\!4\!:\!8\!:\!9),(1\!:\!12\!:\!3\!:\!6),(2\!:\!4\!:\!8\!:\!8),(2\!:\!12\!:\!3\!:\!5),(2\!:\!13\!:\!3\!:\!4),(3\!:\!
5\!:\!8\!:\!6)\}$ \\ \hline
	\end{longtable}
\end{center}
\vspace{-.8cm}

\begin{center}
	\begin{longtable}{|r|c|c|l|l|}
		\caption{Cyclic combinatorial $3$-manifolds that are connected sums. \label{fig:TopTypesCS}} \\
		
		\hline
		$n$ & top. type & $ H_{*} $ & difference cycles of smallest complex* & source \\ \hline 
		\endfirsthead

		\multicolumn{5}{l}%
		{ \tablename\ \thetable{} -- continued from previous page} \\
		\hline 
		$n$ & top. type & $ H_{*} $ & difference cycles of smallest complex & source \\ \hline 
		\endhead

		\hline \multicolumn{5}{r}{{continued on next page --}} \\
		\endfoot
		
		\hline \hline
		\endlastfoot
$12$&$(S^2 \times S^1)^{\#2}$&$(\mathbb{Z},\mathbb{Z}^{2} ,\mathbb{Z}^{2} ,\mathbb{Z})$&$\{(1\!:\!2\!:\!3\!:\!6),(1\!:\!2\!:\!4\!:\!5),(1\!:\!5\!:\!1\!:\!5),(2\!:\!
2\!:\!2\!:\!6),(2\!:\!3\!:\!3\!:\!4)\}$&\cite[Complex $5_{12}$]{Kuehnel85NeighbComb3MfldsDihedralAutGroup} \\ \hline
$16$&$(S^2 \times S^1)^{\#5}$&$(\mathbb{Z},\mathbb{Z}^{5} ,\mathbb{Z}^{5} ,\mathbb{Z})$&$\{(1\!:\!2\!:\!5\!:\!8),(1\!:\!2\!:\!6\!:\!7),(1\!:\!3\!:\!4\!:\!8),(1\!:\!
3\!:\!5\!:\!7),(2\!:\!5\!:\!3\!:\!6),(2\!:\!6\!:\!2\!:\!6),(3\!:\!4\!:\!4\!:\!5)\}$&\cite[Complex $^3 16^{1}_{41}$]{Lutz11TrigMnflds} \\ \hline
$18$&$(S^2 \times S^1)^{\#7}$&$(\mathbb{Z},\mathbb{Z}^{7} ,\mathbb{Z}^{7} ,\mathbb{Z})$&$\{(1\!:\!1\!:\!7\!:\!9),(1\!:\!1\!:\!8\!:\!8),(1\!:\!7\!:\!2\!:\!8),(2\!:\!3\!:\!4\!:\!9),(2\!:\!3\!:\!6\!:\!7),(3\!:\!3\!:\!3\!:\!9),(3\!:\!4\!:\!5\!:\!6),(4\!:\!5\!:\!4\!:\!5)\}$& \\ \hline
$18$&$(S^2 \dtimes S^1)^{\#7}$&$(\mathbb{Z},\mathbb{Z}^{7} ,\mathbb{Z}^{6} \oplus \mathbb{Z}_{2},0)$&$\{(1\!:\!1\!:\!7\!:\!9),(1\!:\!1\!:\!9\!:\!7),(1\!:\!7\!:\!
1\!:\!9),(2\!:\!3\!:\!4\!:\!9),(2\!:\!3\!:\!6\!:\!7),(3\!:\!3\!:\!3\!:\!9),(3\!:\!4\!:\!5\!:\!6),(4\!:\!5\!:\!4\!:\!5)\}$& \\ \hline
$20$&$(S^2 \times S^1)^{\#6}$&$(\mathbb{Z},\mathbb{Z}^{6} ,\mathbb{Z}^{6} ,\mathbb{Z})$&$\{(1\!:\!1\!:\!3\!:\!15),(1\!:\!1\!:\!4\!:\!14),(1\!:\!3\!:\!5\!:\!11),(
1\!:\!5\!:\!5\!:\!9),(1\!:\!8\!:\!2\!:\!9),(2\!:\!3\!:\!7\!:\!8),(2\!:\!4\!:\!5\!:\!9),(3\!:\!5\!:\!5\!:\!7),(3\!:\!7\!:\!3\!:\!7)\}$& \\ \hline
$20$&$(S^2 \dtimes S^1)^{\#6}$&$(\mathbb{Z},\mathbb{Z}^{6} ,\mathbb{Z}^{5} \oplus \mathbb{Z}_{2},0)$&$\{(1\!:\!1\!:\!3\!:\!15),(1\!:\!1\!:\!8\!:\!10),(1\!:\!3\!:\!
7\!:\!9),(1\!:\!4\!:\!6\!:\!9),(1\!:\!8\!:\!5\!:\!6),(1\!:\!9\!:\!4\!:\!6),(2\!:\!3\!:\!5\!:\!10),(3\!:\!5\!:\!5\!:\!7),(3\!:\!7\!:\!3\!:\!7)\}$& \\ \hline
$20$&$(S^2 \times S^1)^{\#4}$&$(\mathbb{Z},\mathbb{Z}^{4} ,\mathbb{Z}^{4} ,\mathbb{Z})$&$\{(1\!:\!2\!:\!2\!:\!15),(1\!:\!2\!:\!4\!:\!13),(1\!:\!4\!:\!5\!:\!10),(
1\!:\!6\!:\!4\!:\!9),(1\!:\!9\!:\!1\!:\!9),(2\!:\!2\!:\!4\!:\!12),(2\!:\!6\!:\!9\!:\!3),(3\!:\!4\!:\!4\!:\!9),(4\!:\!5\!:\!5\!:\!6)\}$& \\ \hline
$20$&$(S^2 \dtimes S^1)^{\#9}$&$(\mathbb{Z},\mathbb{Z}^{9} ,\mathbb{Z}^{8} \oplus \mathbb{Z}_{2},0)$&$\{(1\!:\!2\!:\!7\!:\!10),(1\!:\!2\!:\!8\!:\!9),(1\!:\!4\!:\!
5\!:\!10),(1\!:\!4\!:\!11\!:\!4),(1\!:\!10\!:\!5\!:\!4),(2\!:\!6\!:\!2\!:\!10),(2\!:\!6\!:\!6\!:\!6),(2\!:\!7\!:\!3\!:\!8),(3\!:\!7\!:\!3\!:\!
7)\}$& \\ \hline
$21$&$(S^2 \times S^1)^{\#12}$&$(\mathbb{Z},\mathbb{Z}^{12} ,\mathbb{Z}^{12} ,\mathbb{Z})$&$\{(1\!:\!2\!:\!4\!:\!14),(1\!:\!2\!:\!11\!:\!7),(1\!:\!6\!:\!3\!:\!11),(
1\!:\!9\!:\!4\!:\!7),(2\!:\!4\!:\!7\!:\!8),(3\!:\!3\!:\!3\!:\!12),(3\!:\!4\!:\!5\!:\!9),(3\!:\!6\!:\!7\!:\!5),(3\!:\!9\!:\!4\!:\!5)\}$& \\ \hline
$22$&$(S^2 \dtimes S^1)^{\#12}$&$(\mathbb{Z},\mathbb{Z}^{12} ,\mathbb{Z}^{11} \oplus \mathbb{Z}_{2},0)$&$\{(1\!:\!1\!:\!9\!:\!11),(1\!:\!1\!:\!10\!:\!10),(1\!:\!
9\!:\!2\!:\!10),(2\!:\!3\!:\!6\!:\!11),(2\!:\!3\!:\!8\!:\!9),(3\!:\!4\!:\!4\!:\!11),(3\!:\!4\!:\!11\!:\!4),(3\!:\!6\!:\!5\!:\!8),(3\!:\!11\!:\!4\!:\!4),(
5\!:\!6\!:\!5\!:\!6)\}$& \\ \hline

	\end{longtable}
\end{center}
\vspace{-.8cm}

\tiny
\begin{center}
	\begin{longtable}{|r|c@{}|@{}c@{}|@{}c@{}|@{}l|}
		\caption{Cyclic combinatorial $3$-dimensional graph manifolds$^{**}$ including the homology sphere $\operatorname{HS}$. \label{fig:TopTypesOther}} \\
		
		\hline
		$n$ & top. type & $ H_{*} $ & $\operatorname{TV}(7,1)^{***}$ & difference cycles of smallest complex* \\ \hline 
		\endfirsthead

		\multicolumn{4}{l}%
		{ \tablename\ \thetable{} -- continued from previous page} \\
		\hline 
		$n$ & top. type & $ H_{*} $ & $\operatorname{TV}(7,1)^{***}$ & difference cycles of smallest complex* \\ \hline 
		\endhead

		\hline \multicolumn{4}{r}{{continued on next page --}} \\
		\endfoot
		
		\hline \hline
		\endlastfoot
$20$&$\operatorname{SFS}[D:(3,1)(3,1)] \cup_{m} \operatorname{SFS}[D:(3,1)(3,1)], m = \left ( \begin{array}{cc} -4 & 5 \\ -3 & 4 \end{array} \right )$&$(\mathbb{Z},\mathbb{Z}_{3} \oplus \mathbb{Z}_{3},0,\mathbb{Z})$&$0.0750935889735$&$\{(1\!:\!1\!:\!3\!:\!15),(1\!:\!1\!:\!4\!:\!14),(1\!:\!3\!:\!4\!:\!12),(1\!:\!5\!:\!2\!:\!12),(2\!:\!3\!:\!6\!:\!9),(2\!:\!4\!:\!9\!:\!5),(2\!:\!9\!:\!3\!:\!6),(3\!:\!4\!:\!4\!:\!9)\}$ \\ \hline
$22$&$\operatorname{HS}$&$(\mathbb{Z},0,0,\mathbb{Z})$&$0.0213064178104$&$\{(1\!:\!1\!:\!1\!:\!19),(1\!:\!2\!:\!4\!:\!15),(1\!:\!4\!:\!8\!:\!9),(1\!:\!4\!:\!15\!:\!2),(1\!:\!6\!:\!6\!:\!
9),(2\!:\!4\!:\!10\!:\!6),(2\!:\!5\!:\!6\!:\!9),(2\!:\!9\!:\!2\!:\!9),(2\!:\!9\!:\!5\!:\!6),(4\!:\!4\!:\!4\!:\!10)\}$ \\ \hline
$22$&$\operatorname{SFS}[D: (2,1) (2,1)] \cup_{m} \operatorname{SFS}[D: (2,1) (3,1)], m = \left ( \begin{array}{cc} -5 & 11 \\ -4 & 9 \end{array} \right )$&$(\mathbb{Z},\mathbb{Z}_{3},0,\mathbb{Z})$&$1.55495813209$&$\{(1\!:\!1\!:\!1\!:\!19),(1\!:\!2\!:\!5\!:\!14),(1\!:\!4\!:\!3\!:\!14),(1\!:\!4\!:\!15\!:\!2),(
3\!:\!4\!:\!6\!:\!9),(3\!:\!5\!:\!3\!:\!11),(3\!:\!6\!:\!4\!:\!9),(3\!:\!6\!:\!9\!:\!4),(3\!:\!8\!:\!3\!:\!8),(4\!:\!6\!:\!6\!:\!6)\}$ \\ \hline
$22$&$\operatorname{SFS}[D: (2,1) (3,1)] \cup_{m} \operatorname{SFS}[D: (3,1) (3,1)], m = \left ( \begin{array}{cc} -8 & 11 \\ -5 & 7 \end{array} \right )$&$(\mathbb{Z},\mathbb{Z} ,\mathbb{Z} ,\mathbb{Z})$&$1.87111923986$&$\{(1\!:\!1\!:\!1\!:\!19),(1\!:\!2\!:\!5\!:\!14),(1\!:\!7\!:\!12\!:\!2),(2\!:\!4\!:\!5\!:\!
11),(2\!:\!4\!:\!11\!:\!5),(2\!:\!5\!:\!4\!:\!11),(2\!:\!8\!:\!2\!:\!10),(2\!:\!12\!:\!3\!:\!5),(4\!:\!5\!:\!4\!:\!9),(5\!:\!6\!:\!5\!:\!6)\}$ \\ \hline
$22$&$\operatorname{SFS}[A:(2,1)(2,1)] / m, m = \left ( \begin{array}{cc} 1 & -11 \\ 1 & -10 \end{array} \right )$&$(\mathbb{Z},\mathbb{Z}^{2} ,\mathbb{Z} \oplus \mathbb{Z}_{2},0)$&$23.3297487925$&$\{(1\!:\!1\!:\!4\!:\!16),(1\!:\!1\!:\!11\!:\!
9),(1\!:\!4\!:\!8\!:\!9),(1\!:\!5\!:\!6\!:\!10),(2\!:\!4\!:\!10\!:\!6),(2\!:\!9\!:\!2\!:\!9),(2\!:\!9\!:\!5\!:\!6),(3\!:\!4\!:\!3\!:\!12),(3\!:\!4\!:\!
8\!:\!7),(3\!:\!7\!:\!8\!:\!4)\}$ \\ \hline
	\end{longtable}
\end{center}

\begin{center}
	\begin{longtable}{|l|c|c|c|c|c|c|c|c|c|c|c|c|c|c|c|c|c|c|}
		\caption[Topological types of cyclic $3$-manifolds ordered by number of vertices $n$]{Topological types of cyclic combinatorial $3$-manifolds ordered by number of vertices $n$. \label{fig:top}} \\
		
		\hline
		topological type & $n = 5$ & $6$ & $7$ & $8$ & $9$ & $10$ & $11$ & $12$ & $13$ & $14$ & $15$ & $16$ & $17$ & $18$ & $19$ & $20$ & $21$ & $22$ 
		\\ \hline 
		\endfirsthead

		\multicolumn{19}{l}%
		{ \tablename\ \thetable{} -- continued from previous page} \\
		\hline 
		topological type & $n = 5$ & $6$ & $7$ & $8$ & $9$ & $10$ & $11$ & $12$ & $13$ & $14$ & $15$ & $16$ & $17$ & $18$ & $19$ & $20$ & $21$ & $22$ 
		\\ 
		\hline 
		\endhead

		\hline \multicolumn{19}{r}{{continued on next page --}} \\
		\endfoot

		\hline \hline
		\endlastfoot
$S^3$&$ \times $&$ \times $&$ \times $&$ \times $&$ \times $&$ \times $&$ \times $&$ \times $&$ \times $&$ \times $&$ \times $&$ \times $&$ \times $&$ \times $&$ \times $&$ \times $&$ \times $&$ \times $\\ \hline
$S^3/ \operatorname{SL}(2,3)$&&&&&&&&&&&&$ \times $&&&&&&$ \times $\\ \hline
$\Sigma^3$&&&&&&&&&&&&&$ \times $&&&&&$ \times $\\ \hline
$P_2 =S^3/Q_8$&&&&&&&&&&&$ \times $&&&$ \times $&&&$ \times $&\\ \hline
$P_4 =S^3/Q_{16}$&&&&&&&&&&&&&&&&&&$ \times $\\ \hline
$P_7 =S^3/Q_{28}$&&&&&&&&&&&&&&&&$ \times $&&\\ \hline
$P_8 =S^3/Q_{32}$&&&&&&&&&&&&&&&&&&$ \times $\\ \hline
$L(3,1)$&&&&&&&&&&$ \times $&&$ \times $&&&&$ \times $&$ \times $&$ \times $\\ \hline
$L(5,1)$&&&&&&&&&&&&&&$ \times $&&&&$ \times $\\ \hline
$L(7,1)$&&&&&&&&&&&&&&&&&&$ \times $\\ \hline
$L(8,3)$&&&&&&&&&&&&&&$ \times $&&$ \times $&&$ \times $\\ \hline
$L(15,4)$&&&&&&&&&&&&&&&&&&$ \times $\\ \hline
$S^2 \times S^1$&&&&&&$ \times $&&$ \times $&&$ \times $&&$ \times $&&$ \times $&&$ \times $&&$ \times $\\ \hline
$S^2 \dtimes S^1$&&&&&$ \times $&$ \times $&$ \times $&$ \times $&$ \times $&$ \times $&$ \times $&$ \times $&$ \times $&$ \times $&$ \times $&$ \times $&$ \times $&$ \times $\\ \hline
$\mathbb{R}P^2 \times S^1$&&&&&&&&&&&&&$ \times $&&$ \times $&$ \times $&$ \times $&$ \times $\\ \hline
$\mathbb{R}P^3$&&&&&&&&&&&$ \times $&&$ \times $&$ \times $&&&$ \times $&$ \times $\\ \hline
$\mathbb{T}^3$&&&&&&&&&&&$ \times $&$ \times $&$ \times $&$ \times $&$ \times $&$ \times $&$ \times $&$ \times $\\ \hline
$\mathfrak{B}_2$&&&&&&&&&&&&$ \times $&$ \times $&$ \times $&$ \times $&$ \times $&$ \times $&$ \times $\\ \hline
$\mathfrak{B}_4$&&&&&&&&&&&&&&$ \times $&&$ \times $&&$ \times $\\ \hline
$\mathfrak{G}_2$&&&&&&&&&&&&&&&&$ \times $&&\\ \hline
$\mathfrak{G}_3$&&&&&&&&&&&&&&&&&$ \times $&\\ \hline
$\mathbb{K}^2 \times S^1$&&&&&&&&&&&&&&$ \times $&&$ \times $&$ \times $&$ \times $\\ \hline
$\Sigma(2,3,7)$&&&&&&&&&&&&&&&$ \times $&&&\\ \hline
$(S^2 \times S^1)^{\#2}$&&&&&&&&$ \times $&&&&&&&&&&\\ \hline
$(S^2 \times S^1)^{\#4}$&&&&&&&&&&&&&&&&$ \times $&&\\ \hline
$(S^2 \times S^1)^{\#5}$&&&&&&&&&&&&$ \times $&&&&&&\\ \hline
$(S^2 \times S^1)^{\#6}$&&&&&&&&&&&&&&&&$ \times $&&\\ \hline
$(S^2 \dtimes S^1)^{\#6}$&&&&&&&&&&&&&&&&$ \times $&&\\ \hline
$(S^2 \times S^1)^{\#7}$&&&&&&&&&&&&&&$ \times $&&&&\\ \hline
$(S^2 \dtimes S^1)^{\#7}$&&&&&&&&&&&&&&$ \times $&&&&\\ \hline
$(S^2 \dtimes S^1)^{\#9}$&&&&&&&&&&&&&&&&$ \times $&&\\ \hline
$(S^2 \times S^1)^{\#12}$&&&&&&&&&&&&&&&&&$ \times $&$ \times $\\ \hline
$(S^2 \dtimes S^1)^{\#12}$&&&&&&&&&&&&&&&&&&$ \times $\\ \hline
$\operatorname{SFS} [\mathbb{T}^2: (1,1)]$&&&&&&&&&&&&&&$ \times $&&&&\\ \hline
$\operatorname{SFS} [\mathbb{T}^2: (1,7)]$&&&&&&&&&&&&&&&&&$ \times $&\\ \hline
$(\mathbb{R}P^2)^{\#3} \times S^1$&&&&&&&&&&&&&&&&&$ \times $&\\ \hline
$\operatorname{SFS} [D_: (3,1) (3,1)]$&&&&&&&&&&&&&&&&&&$ \times $\\ \hline
$\operatorname{SFS} [\mathbb{K}^2/n2: (1,5)]$&&&&&&&&&&&&&&&&$ \times $&&\\ \hline
$\operatorname{SFS} [\mathbb{K}^2: (2,1) (2,1) (2,1)]$&&&&&&&&&&&&&&&&&$ \times $&\\ \hline
$\operatorname{SFS} [\mathbb{K}^2: (2,1) (2,1)]$&&&&&&&&&&&&&&&&&&$ \times $\\ \hline
$\operatorname{SFS} [\mathbb{K}^2: (2,1)]$&&&&&&&&&&&&&&&&&$ \times $&$ \times $\\ \hline
$\operatorname{SFS} [\mathbb{K}^2: (3,1)]$&&&&&&&&&&&&&&&&&&$ \times $\\ \hline
$\operatorname{SFS} [(\mathbb{R}P^2)^{\#3}: (1,1)]$&&&&&&&&&&&&&&&&&&$ \times $\\ \hline
$\operatorname{SFS} [\mathbb{R}P^2: (2,1) (2,1) (2,1)]$&&&&&&&&&&&&&&$ \times $&&&$ \times $&\\ \hline
$\operatorname{SFS} [\mathbb{R}P^2: (2,1) (3,1)]$&&&&&&&&&&&&&&&&&&$ \times $\\ \hline
$\operatorname{SFS} [\mathbb{R}P^2: (2,1) (5,1)]$&&&&&&&&&&&&&&&&&&$ \times $\\ \hline
$\operatorname{SFS} [\mathbb{R}P^2: (3,1) (3,1) (3,1)]$&&&&&&&&&&&&&&&&&$ \times $&\\ \hline
$\operatorname{SFS} [\mathbb{R}P^2: (3,1) (3,1) (3,2)]$&&&&&&&&&&&&&&&&&$ \times $&\\ \hline
$\operatorname{SFS} [\mathbb{R}P^2: (3,1) (3,2)]$&&&&&&&&&&&&&&&&$ \times $&&$ \times $\\ \hline
$\operatorname{SFS} [\mathbb{R}P^2: (3,1) (4,3)]$&&&&&&&&&&&&&&&&&&$ \times $\\ \hline
$\operatorname{SFS} [\mathbb{R}P^2: (5,2) (5,3)]$&&&&&&&&&&&&&&&&&&$ \times $\\ \hline
$\operatorname{SFS} [S^2: (2,1) (2,1) (2,1) (3,-5)]$&&&&&&&&&&&&&&&&&$ \times $&\\ \hline
$\operatorname{SFS} [S^2: (2,1) (2,1) (3,1) (3,-2)]$&&&&&&&&&&&&&&&&&&$ \times $\\ \hline
$\operatorname{SFS} [S^2: (2,1) (2,1) (3,1) (3,-4)]$&&&&&&&&&&&&&&&&&&$ \times $\\ \hline
$\operatorname{SFS} [S^2: (3,1) (3,1) (4,-3)]$&&&&&&&&&&&&&&&&$ \times $&&\\ \hline
$\operatorname{SFS} [S^2: (3,1) (3,1) (5,-3)]$&&&&&&&&&&&&&&&&&&$ \times $\\ \hline
$\operatorname{SFS} [S^2: (3,1) (3,1) (9,-7)]$&&&&&&&&&&&&&&&&&&$ \times $\\ \hline
$\operatorname{SFS} [S^2: (3,2) (3,2) (3,-1)]$&&&&&&&&&&&&&&&&&$ \times $&\\ \hline
$\operatorname{SFS} [S^2: (3,2) (4,1) (4,-3)]$&&&&&&&&&&&&&&&&&&$ \times $\\ \hline
$\operatorname{SFS} [S^2: (4,1) (4,1) (4,-3)]$&&&&&&&&&&&&&&&&&$ \times $&\\ \hline
$\operatorname{SFS} [S^2: (4,1) (5,2) (5,-3)]$&&&&&&&&&&&&&&&&&&$ \times $\\ \hline
$\operatorname{SFS} [S^2: (5,1) (5,1) (5,-4)]$&&&&&&&&&&&&&&&&&$ \times $&\\ \hline
$\operatorname{SFS}[D:(3,1)(3,1)] \cup_{m} \operatorname{SFS}[D:(3,1)(3,1)], m = \left ( \begin{array}{cc} -4 & 5 \\ -3 & 4 \end{array} \right )$&&&&&&&&&&&&&&&&$ \times $&&\\ \hline
$\operatorname{HS}$&&&&&&&&&&&&&&&&&&$ \times $\\ \hline
$\operatorname{SFS}[D: (2,1) (2,1)] \cup_{m} \operatorname{SFS}[D: (2,1) (3,1)], m = \left ( \begin{array}{cc} -5 & 11 \\ -4 & 9 \end{array} \right )$&&&&&&&&&&&&&&&&&&$ \times $\\ \hline
$\operatorname{SFS}[D: (2,1) (3,1)] \cup_{m} \operatorname{SFS}[D: (3,1) (3,1)], m = \left ( \begin{array}{cc} -8 & 11 \\ -5 & 7 \end{array} \right )$&&&&&&&&&&&&&&&&&&$ \times $\\ \hline
$\operatorname{SFS}[A:(2,1)(2,1)] / m, m = \left ( \begin{array}{cc} 1 & -11 \\ 1 & -10 \end{array} \right )$&&&&&&&&&&&&&&&&&&$ \times $\\ \hline
	\end{longtable}
\end{center}

\vspace{-.7cm}
\noindent
$^*$ The smallest complex is the lexicographically (with respect to the difference cycles) minimal complex of all complexes of a given topological type with the smallest number of vertices.

\noindent
$^{**}$ Graph manifolds consist of Seifert fibered spaces with toroidal boundary components, glued together along homeomorphisms of the boundary components given by an element of the mapping class group of the torus $m \in \operatorname{SL}(2,\mathbb{Z})$. $D$ denotes a disc and $A$ an annulus.

\noindent
$^{***}$ The symbol $\operatorname{TV}(7,1)$ denotes the Turaev-Viro invariant (see \cite{Turaev92TuraevViroInvariant}) with parameters \texttt{r = 7} and \texttt{whichRoot = 1} as indicated in the documentation of \texttt{regina}.
\end{landscape}

\normalsize
	Such a manifold is determined by how the canonical meridian $m^{+}$ and longitude $\ell^{+}$ of $\operatorname{HS}^{+}$ are glued to the boundary of 
	$\operatorname{HS}^{-}$. Following the conventions of \cite{Saveliev02InvariantsOfHomology3Spheres} a gluing is given by integers $a$, $b$, $c$ and $d$
	with $ad -bc = -1$ and $\ell^{+}$ is glued to $a \cdot \ell^{-} + c \cdot m^{-}$ and $m^{+}$ is glued to $b \cdot \ell^{-} + d \cdot m^{-}$ where $c = \pm 1$ whenever the
	resulting manifold is a homology sphere. The canonical meridians of both $\operatorname{HS}^{+}$ and 
	$\operatorname{HS}^{-}$ were computed using \texttt{simpcomp} and are given by the thick lines in Figure \ref{fig:homSphere}. It follows that
	$m^{+}$ is glued to $-m^{-}$, thus $b=0$, $d = -1$, and as a result $a = 1$. This determines the topological type of $\operatorname{HS}$.

	Alternatively, using Matveev's \texttt{Three-manifold Recognizer} \cite{Matveev13Recognizer} $\operatorname{HS}$ is identified as the graph manifold
	$$ \operatorname{SFS}[D: (2,1) (3,1)] \cup_{m} \operatorname{SFS}[D: (2,1) (3,1)], m = \left ( \begin{array}{cc} -10 & 11 \\ -9 & 10 \end{array} \right ), $$
	here given in the notation used by \texttt{regina}.

	\medskip
	All complexes with fewer than 18 vertices have already been described in literature. See the indicated sources in Table \ref{fig:TopTypesSpherical}, \ref{fig:TopTypesS2xR}, \ref{fig:TopTypesFlat} and \ref{fig:TopTypesCS}. The remaining topological types of cyclic $3$-manifolds were identified using \texttt{regina}, the (orientable) graph manifolds were additionally checked using the \texttt{Three-manifold Recognizer}. The notation for the Seifert fibered spaces as well as the graph manifolds is following the one \texttt{regina} is using which in turn is based on work by Burton \cite{Burton07EnumNonOr3Mflds} and Orlik \cite[pg. 88]{Orlik72SeifertMflds} (note that the \texttt{Three-manifold Recognizer} is using a slightly different notation). To make sure that none of the Seifert fibered spaces or graph manifolds equal any other topological type of combinatorial $3$-manifold previously described in the classification, we additionally computed the Turaev-Viro invariant of the manifolds (see \cite{Turaev92TuraevViroInvariant}) whenever necessary. See the documentation of \texttt{regina} or one of the indicated sources for more information.

It is interesting to see that some of the homological types of the complexes do not occur for certain integers. Especially, if $n$ is a prime number, the number of topologically distinct complexes seems to be limited. In particular, we believe the following to be true.

\begin{conj}
	Let $M$ be a combinatorial $3$-manifold with transitive cyclic symmetry homeomorphic to $S^2 \times S^1$. Then $M$ has an even number of vertices.
\end{conj}

\section{Further results}
\label{sec:lensspace}

A direct consequence from the extension of the classification of transitive cyclic combinatorial manifolds together with Theorem \ref{main} is the following result.

\begin{kor}
	\label{kor:denseSeries}
	There are exactly $396$ combinatorially distinct dense infinite families of combinatorial $3$-manifolds starting with a triangulation with fewer than $23$ vertices.
\end{kor}

The results from Section \ref{sec:inf} allow us to formulate a number of further results similar to Corollary \ref{kor:denseSeries} using the data of the classification. However, these infinite families typically only contain a few distinct topological types of $3$-manifolds: Most of the infinite families have members of type $S^2 \times S^1$ or $S^2 \dtimes S^1$, family number $17$ (\texttt{SCSeriesK(17,k)} in \textsf{simpcomp}) has members of type $\mathbb{T}^3$ and $\mathfrak{B}_2$, and families number $30$, $42$ and $356$ (\texttt{SCSeriesK(30,k)}, \texttt{SCSeriesK(42,k)} and \texttt{SCSeriesK(356,k)} in \textsf{simpcomp}) contain combinatorial $3$-manifolds of three further topological types. This observation is emphasized by the following upper bound on the Betti numbers of the members of a dense infinite family of combinatorial $3$-manifolds.

\begin{prop}
	\label{prop:boundedBetti}
	Let $M_k = \{ d_{1,k} , \ldots , d_{m,k} \}$ with $d_{i,k} = ( a_i^0 : a_i^1 : a_i^2 : a_i^{3} + k )$ be a dense infinte family of transitve cyclic combinatorial $3$-manifolds with $n+k$ vertices. Then there exists a constant $c$ such that
	$$ \sum \limits_{i=0}^{d} \beta_i (M_0) + c \geq \sum \limits_{i=0}^{d} \beta_i (M_k) $$
	for all $k \geq 0$.
\end{prop}

\begin{proof}
	By Kuiper's discrete Morse relations (see Section \ref{sec:prelims} or \cite{Kuiper71MorseRelations}) we know that the number of critical points of any rsl-function on $M_k$ is an upper bound for the sum of the Betti numbers of $M_k$. We will prove Proposition \ref{prop:boundedBetti} by giving an upper bound on the number of critical points of the rsl-function $f_k$ on $M_k$ induced by the natural ordering of the vertices $V_k := \{ 0 , 1, \ldots , n+k-1 \}$ of $M_k$.

	In this setting, the relevant slicings $S_i$ of $M_k$ with respect to $f_k$ are given by the partitions $P_i = (\{ 0, 1, \ldots , i-1 \},\{ i , \ldots , n+k-1\})$ of $V_k$. Now, observe that the set of tetrahedra having both vertices in $\{ 0, 1, \ldots , i-1 \}$ and $\{ i , \ldots , n+k-1\}$ must be isomorphic to the set of tetrahedra having vertices in both $\{ 0, 1, \ldots , i \}$ and $\{ i+1 , \ldots , n+k-1\}$ for all $\mu < i < n+k-1-\mu$ where $\mu := \underset{1 \leq i \leq m}{\operatorname{max}} (a_i^0 + a_i^1 + a_i^2)$. This follows from the cyclic symmetry and the fact that by construction all tetrahedra of $M_k$ for all $k \geq 0$ have their vertices within an interval (modulo $(n+k)$) of length less or equal $\mu$. As a consequence, we have $S_i = S_{i+1}$ and since a vertex of a combinatorial manifold can only be critical if the topological type of the associated slicing changes when the vertex is passed, none of the vertices $\mu < i < n+k-1-\mu$ can be critical. 

	Moreover, note that whenever $\mu < n+k-1-\mu$ the combinatorial types of the remaining slicings $S_i$, $1 \leq i \leq \mu+1$ and $n+k-1-\mu \leq i \leq n+k-1$ in $M_k$ are independent of $k$ and hence the number of critical points of $f_k$ is the same for all $k$. Now since $\mu < \frac{n}{2}$ we have $\mu < n+k-1-\mu$ whenever $k \geq 1$ and the number of critical points of $f_k$ for any $k$ is bounded above by the maximum of the number of critical points of $M_0$ and $M_1$. Hence, for $c$ being the maximum of the number of critical points of $M_0$ and $M_1$ the statement follows.
\end{proof}

\bigskip
In order to find infinite families which are richer from the topological point of view, we want to use the classification of transitive cyclic combinatorial $3$-manifolds described in Section \ref{sec:class} to search for manifolds that look like the start of an infinite family of {\em $2$-neighborly} combinatorial $3$-manifolds (which thus have to contain an {\em increasing number of difference cycles}) containing {\em infinitely many} members of pairwise distinct topological types.

This is motivated by the situation in dimension $2$ where several of such infinite families exist. There is a family of neighborly orientable surfaces of genus $\frac{1}{6} {12s + 4 \choose 2}$ with $12s + 7$ vertices (cf. \cite[Fig. 2.15]{Ringel74MapColThm} and \cite[Example 2.7]{Kuhnel96CentrSymmTightSurf}) starting with the $7$-vertex M\"obius torus. In addition, many further families of transitive combinatorial $2$-manifolds with similar properties can be found in \cite{Lutz09EquivdCovTrigsSurf} by Lutz.

Note that there are infinite $2$-neighborly families of combinatorial $3$-manifolds described in the literature. However, these families have members of only a constant number of distinct topological types per family (see the boundary of the cyclic $4$-polytopes for a family of neighborly $S^3$, \cite{Kuehnel85NeighbComb3MfldsDihedralAutGroup} and \cite[Section 4.2]{Spreer10Diss} for families of sphere bundles over the circle, and \cite{Brehm09LatticeTrigE33Torus} and \cite{Kuehnel96PermDiffCyc} for neighborly $3$-dimensional tori).

%

\medskip
A detailed analysis of the data provided by the classification led to a general construction principle for infinite families of Seifert Fibred Spaces and in particular Brieskorn homology spheres with transitive cyclic symmetry and an infinite number of distinct topological types per family. However, for the remainder of this article we will focus on an infinite family of topologically distinct lens spaces which was conjectured following a different approach and refer the reader to \cite{Spreer12VarCyclicPolytopeCompExpI} where the former type of infinite family is described in detail.

\begin{satz}
	\label{thm:lensSeries}
  The complex
  \footnotesize
  \begin{eqnarray}
    \label{eq:lensSeries}
    L_k &:=& \left \{ \, (1:1:1:11+4k), (1:2:4:7+4k), (1:4:2:7+4k), (1:4:7+4k:2) \, \right \} \nonumber \\
        && \bigcup \limits_{i=0}^{k} \,\, \left \{ \, (2:5+2i:2:5+4k-2i) , (4:2+2i:4:4+4k-2i) \, \right \}
  \end{eqnarray}
  \normalsize
  is a combinatorial $3$-manifold with $n = 14+4k$, $k \geq 0$, vertices. It is  homeomorphic to the lens space $L(k^2+4k+3, k+2)$.
\end{satz}

\begin{proof}
Obviously, $L_k$ has $n = 14+4k$ vertices. By looking at Figure \ref{fig:lensLink} we can verify that the link $\operatorname{lk}_{L_k} (0)$ of the vertex $0$ in $L_k$ is a triangulated $2$-sphere. Hence, as $L_k$ has transitive symmetry, it follows immediately that $L_k$ is in fact a combinatorial $3$-manifold for all $k \geq 0$. Furthermore, we can see that $\operatorname{lk}_{L_k} (0)$ has $13+4k$ vertices and thus $L_k$ is $2$-neighborly.
\begin{figure}
	\begin{center}
	 \includegraphics[width=0.6\textwidth]{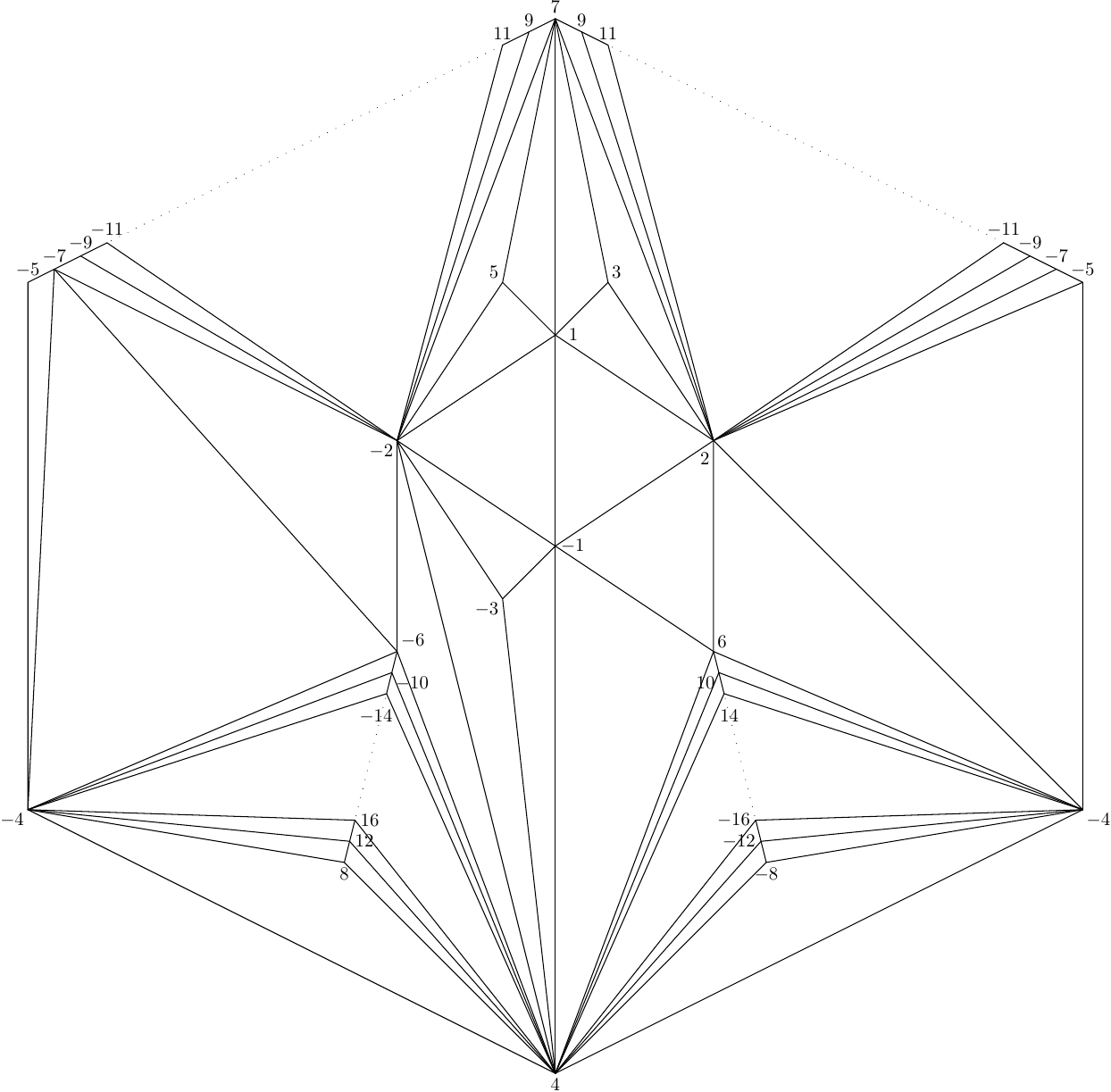}
	 \caption{Link of vertex $0$ of $L_k$ -- a triangulated $2$-sphere with $13+4k$ vertices.\label{fig:lensLink}}
	\end{center}
\end{figure}

\medskip
To determine the exact topological type of $L_k$ we will proceed as follows:
\begin{enumerate}
	\item for all $k \geq 0$, determine a Heegaard splitting of $L_k$ of genus $1$ which will be given by a decomposition of 
		$L_k$ into two disjoint solid tori $T_k^-$ and $T_k^+$ and a center piece of type cartesian product of a torus with an interval,
	\item draw the middle torus $S_k$ of the center piece as a slicing (see Figure \ref{fig:HeegaardDiagram}),
	\item choose a basis $\{\alpha_k^-, \beta_k^-\}$ for the $1$-homology of the boundary of the solid torus $T_k^-$ such that $H_1 (T_k^-) = \langle \beta_k^- \rangle$,
	\item do the same for the solid torus $T_k^+$, that is choose $\{ \alpha_k^+, \beta_k^+ \}$ such that $H_1 (\partial T_k^+) = \langle \alpha_k^+, \beta_k^+ \rangle$ and $H_1 (T_k^+) = \langle \beta_k^+ \rangle$,
	\item with the help of the slicing $S_k$, determine the class of $\alpha_k^-$ in $H_1 (\partial T_k^+)$ -- by construction this will be a torus knot which will determine the topological type of $L_k$.
\end{enumerate}

{\bf 1.} For all $k \geq 0$, the span of the even labeled vertices $T_k^- := \operatorname{span} (\{ 0, 2, \ldots , n-1 \})$ as well as the span of the odd labeled vertices $T_k^+ := \operatorname{span} (\{ 1, 3, \ldots , n \})$ (which is combinatorially isomorphic to $T_k^-$ by the cyclic symmetry) form a solid torus. 

To see this note that $T_k^-$ together with $T_k^+$ are exactly the difference cycles
$$ T_k^- \cup T_k^+ = \bigcup \limits_{i=0}^{k} \,\, \left \{ \, (4:2+2i:4:4+4k-2i) \, \right \} \subset L_k. $$
Since the gcd of $4$, $2+2i$ and $4+4k-2i$, $0 \leq i \leq k$, is $2$ for all $k \geq 0$, $T_k^-$ and $T_k^+$ are disjoint but connected and we have
$$ T_k^- \cong T_k^+ \cong \bigcup \limits_{i=0}^{k} \,\, \left \{ \, (2:1+i:2:2+2k-i) \, \right \} =: T_k . $$
For $k=0$ we have $T_0 = \{ (1:1:1:4) \} \cong B^2 \times S^1$. If $k \geq 1$, $T_k$ consists of $k+1$ difference cycles which we will denote by $\delta_i := (2:1+i:2:2+2k-i)$, $0 \leq i \leq k$. $\delta_i$ shares two triangles per tetrahedron with $\delta_{2+i}$, $0 \leq i \leq k-2$, $\delta_{k-1}$ shares two triangles per tetrahedron with $\delta_k$, $k \geq 1$, $\delta_1$ shares two triangles per tetrahedron with itself and $\delta_0$ shares two triangles per tetrahedron with $\partial T_k$ and hence contains the complete boundary of $T_k$. Altogether, we have the following collapsing sheme of $T_k$:

\bigskip
\noindent
\includegraphics[width=1.0\textwidth]{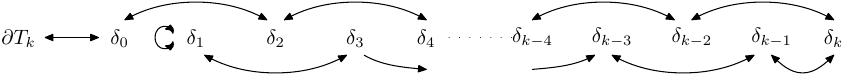}

\bigskip
\noindent
Thus, $T_k$ collapses onto $\delta_1 = (2:2:2:1+2k)$ and since $\delta_1$ contains an odd number of vertices we have $\delta_1 \cong (1:1:1:4+2k) \cong  B^2 \times S^1$. 

Finally $T_k^-$ and $T_k^+$ span all the vertices of $L_k$ and contain all tetrahedra with only even or only odd vertex labels. As a consequence the center piece between the two solid tori is of type cartesian product of a torus with an interval, $T_k^-$ and $T_k^+$ define a Heegaard splitting of $L_k$ of genus $1$ and $L_k$ is homeomorphic to the $3$-sphere, $S^2 \times S^1$ or a lens space $L(p,q)$.

\medskip
{\bf 2.} Following from the above, the slicing $S_k$ between the odd labeled and the even labeled vertices is a torus and is shown in Figure \ref{fig:HeegaardDiagram}. It is interesting to see that apart from $T_k^-$ and $T_k^+$ the difference cycles $(1:2:4:7+4k)$ and $(1:4:2:7+4k)$ are the only ones which do not contain two odd and two even labels per tetrahedron and thus are the only ones which are not sliced by $S_k$ in a quadrilateral. Hence, $S_k$ consists of only $28+8k$ triangles but $(2+k)(14+4k)+7+2k = 4k^2+24k+35$ quadrilaterals. Its complete $f$-vector is 
$$f(S_k) = (4k^2+28k+49,8k^2+60k+112, (8k+28) \Delta , (4k^2+24k+35) \Box).$$
\begin{figure}[p]
	\begin{center}
	 \includegraphics[width=1.0\textwidth]{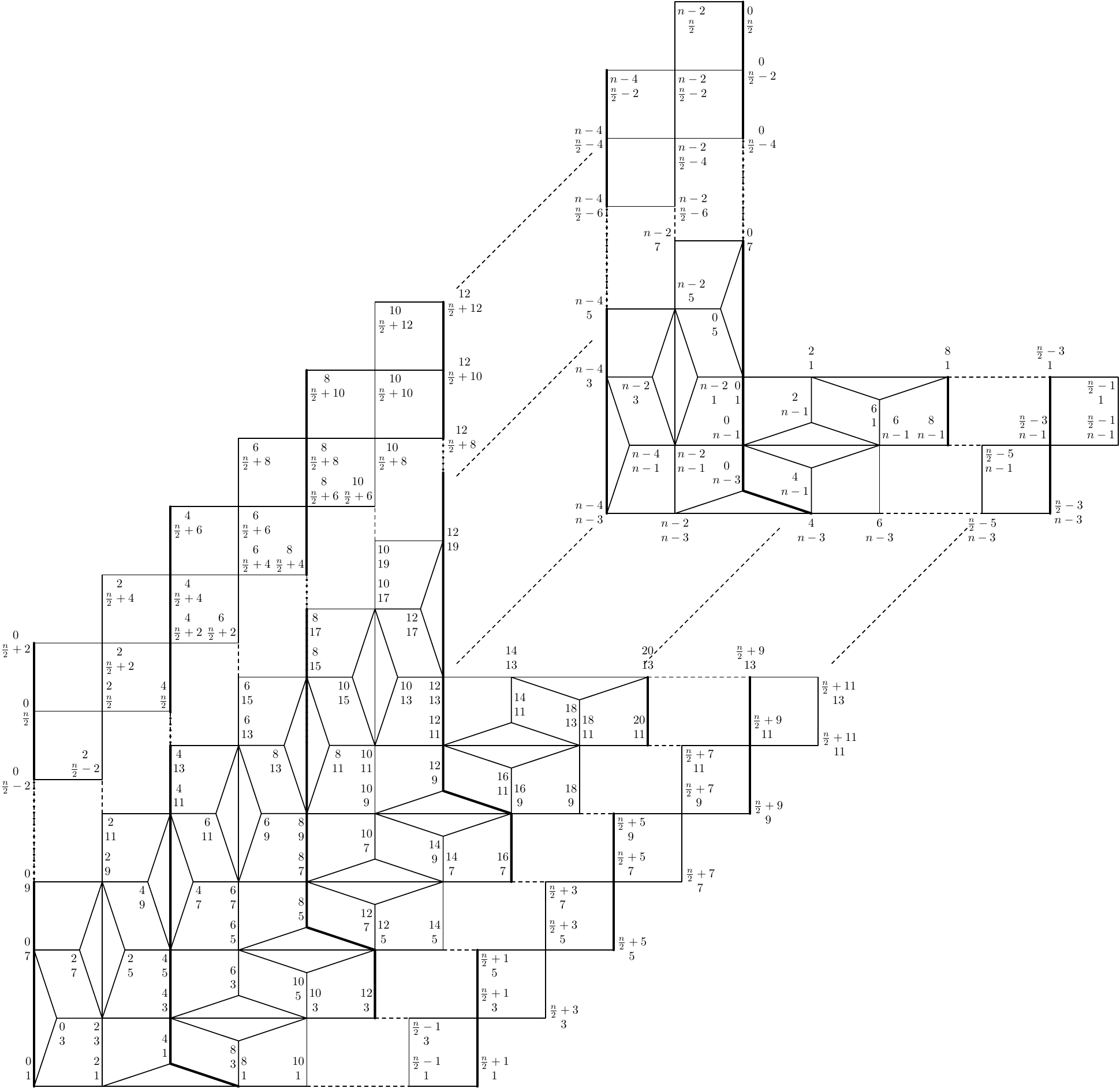}
	 \caption{Slicing of $L_k$ between the odd labeled and the even labeled vertices -- a torus.\label{fig:HeegaardDiagram}}
	\end{center}
\end{figure}

\medskip
{\bf 3.} and {\bf 4.} In order to find a suitable basis of $H_1 (\partial T_k^-)$ as indicated above, let us first take a look at $\partial T_k^-$ itself which is shown in Figure \ref{fig:torus}. 
\begin{figure}
	\begin{center}
	 \includegraphics[width=0.8\textwidth]{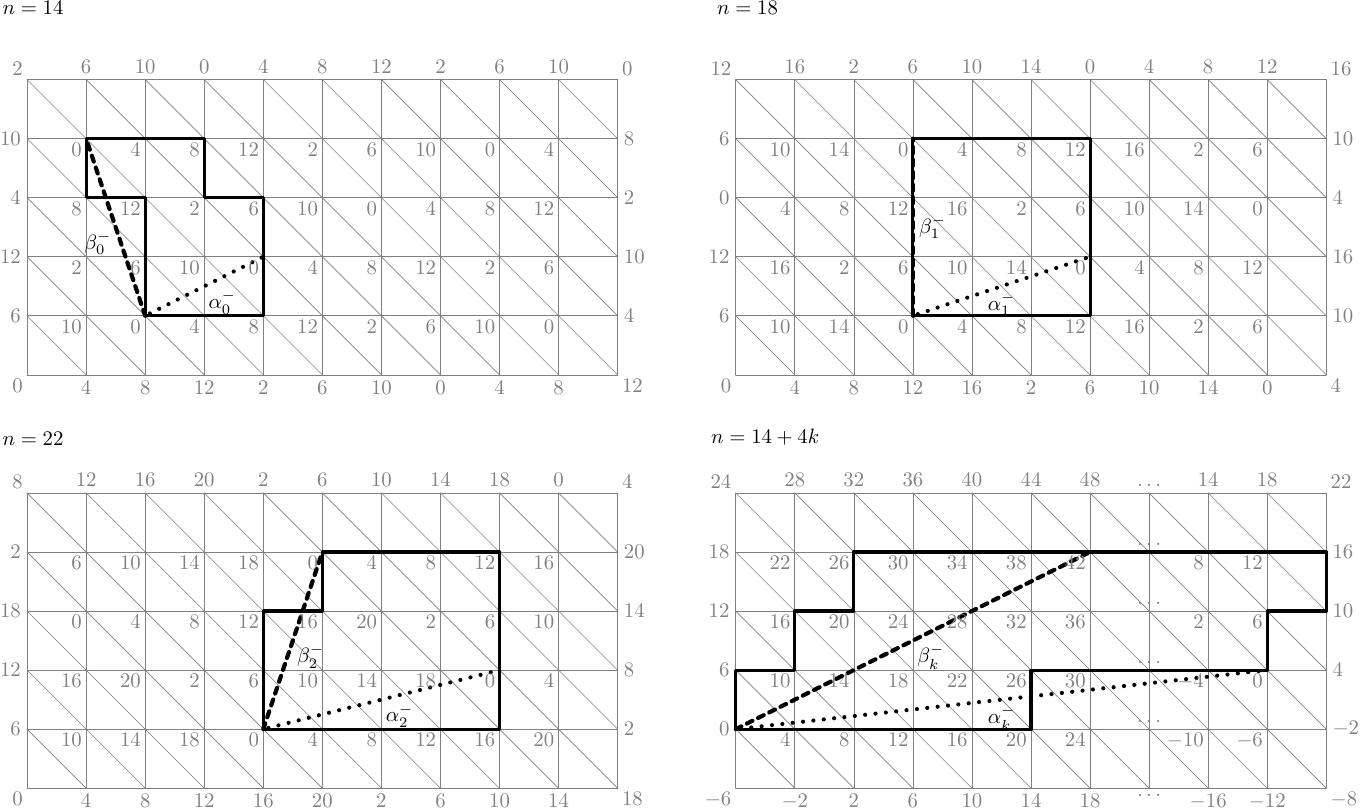}
	 \caption{Fundamental domain of the boundary of $T_k^-$ together with the basis $\{ \alpha_k^-, \beta_k^- \}$ of $H_1 (\partial T_k^-)$ for selected values of $k$ and in greater generality.\label{fig:torus}}
	\end{center}
\end{figure}
We choose the basis $\{ \alpha_k^-, \beta_k^- \}$ of $H_1 (\partial T_k^-)$ to be
\begin{eqnarray}
	\alpha_k^- &=& [ 0, 4, 8, \ldots , n-6, 0 ] \nonumber \\	
	\beta_k^- &=& [ 0, 6, 12, 18, 22, 26, \ldots , n-4, 0 ] \nonumber
\end{eqnarray}
or in the case that $n < 26$ as indicated in Figure \ref{fig:torus}. By construction, $\alpha_k^-$ is contractible in $T_k^-$ and $H_1 (T_k^-) = \langle \beta_k^- \rangle$.

For $\{ \alpha_k^+, \beta_k^+ \}$ we choose analogously
\begin{eqnarray}
	\alpha_k^+ &=& [ 1, 5, 9, \ldots , n-5, 1 ] \nonumber \\	
	\beta_k^+ &=& [ 1, 7, 13, 19, 23, 27, \ldots , n-3, 1 ] \nonumber
\end{eqnarray}
and hence $H_1 (T_k^+) = \langle \beta_k^+ \rangle$.

\medskip
{\bf 5.} To finish the proof we will express $\alpha_k^-$ in terms of $\alpha_k^+$ and $\beta_k^+$. This is done by a map $\phi : H_1 (\partial T_k^-) \to H_1 (\partial T_k^+)$ which lifts any path in $L_k$ passing only even labeled vertices (a path in $\partial T_k^-$) to a homologically equivalent path passing only odd labeled vertices (a path in $\partial T_k^+$). The image of a path under $\phi$ can be determined with the help of the slicing $S_k$. In the case of $\alpha_k^-$ it is the thick line in Figure \ref{fig:HeegaardDiagram} and results in the following path:
\begin{equation}
	\begin{array}{llll}
	\phi (\alpha_k^-) &= [ & n-7, n-9, n-11, \ldots , 9, 7, 1, n-1, n-3, &\\
	&& n-3, n-5, n-7, \ldots , 13, 11, 5, 3, 1, &\\
	&& 1, n-1, n-3 \ldots , 17, 15, 9, 7, 5, &\\
	&& \ldots &\\
	&& n-13, n-15, n-17, \ldots , 3, 1, n-5, n-7 &] .
	\end{array}
\end{equation}
By taking a closer look to Figure \ref{fig:torus} we see that all edges of a path of type $\langle s, s-2 \rangle$ in both $\partial T_k^-$ and $\partial T_k^+$ go from the left upper corner of a square of the grid to the lower right corner ($\searrow$) whereas an edge of type $\langle s, s-6 \rangle$ is simply going down in the grid ($\downarrow$). As $\phi (\alpha_k^-)$ has $(k+2)(2k+2) + 2k+1$ segments of type $\searrow$ and $k+3$ segments of type $\downarrow$, $\phi (\alpha_k^-)$ results in the vector $(2k^2+8k+5,2k^2+9k+8)$ on the integer grid with basis $(\rightarrow, \downarrow)$ (cf. Figure \ref{fig:torus} where $\partial T_k^+$ is obtained from $\partial T_k^-$ by the shift $v \mapsto (v+1) \mod n$ of all vertex labels).

On the other hand, we know that $\alpha_k^+$ corresponds to the vector $(k+2,-1)$ and $\beta_k^+$ to $(k-1,-3)$ on the grid for $\partial T_k^+$ with basis  $(\rightarrow, \downarrow)$. Thus, to express $\phi (\alpha_k^-)$ in terms of $\alpha_k^+$ and $\beta_k^+$ we have to solve the following system of equations:
\begin{equation}
	\begin{array}{lrlrll}
	\operatorname{I.}&(k+2)q &+& (k-1)p&=&2k^2+8k+5 \\ 
	\operatorname{II.}&(-1)q&+&(-3)p&=&2k^2+9k+8
	\end{array}
\end{equation}
which results in the solution
$$ q = k^2+3k+1; \quad p = -k^2-4k-3$$
and hence
$$\phi (\alpha_k^-) = (k^2+3k+1) \alpha_k^+ + (-k^2-4k-3) \beta_k^+.$$
Furthermore, note that $L(p,q_1) \cong L(p,q_2)$ if and only if $q_1 \equiv \pm q_2^{\pm 1} \mod p$ from which it follows that
$$ L_k \cong L(k^2 + 4k + 3,k+2).$$
\end{proof}

The family $L_k$ can be modified into a family of $3$-spheres which only differs from $L_k$ in the part which is disjoint to the slicing $S_k$. Hence, Theorem \ref{thm:lensSeries} shows that combinatorial surgery of infinitely many essentially different types can be applied in a setting respecting the cyclic symmetry of the underlying combinatorial manifolds. The following corollary, which is a direct implication of Theorem \ref{thm:lensSeries}, summarizes the findings of this section under a more general point of view. 

\begin{kor}
	There are infinitely many topologically distinct combinatorial (prime) $3$-manifolds with transitive cyclic symmetry.
\end{kor}

\medskip
The author wants to thank Wolfgang K\"uhnel, Ben Burton and the reviewer for countless helpful comments.


\begin{thebibliography}{99}

\bibitem{Altshuler71PolyhedralRealizationsTori}
A.~Altshuler.
\newblock {P}olyhedral realization in {$R^{3}$} of triangulations of the torus
  and {$2$}-manifolds in cyclic {$4$}-polytopes.
\newblock {\em Discrete Math.}, 1(3):211--238, 1971/1972.

\bibitem{Altshuler74NeighComb3Mnf9Vert}
A.~Altshuler and L.~Steinberg.
\newblock {N}eighborly combinatorial {$3$}-manifolds with {$9$} vertices.
\newblock {\em Discrete Math.}, 8:113--137, 1974.

\bibitem{Banchoff67CritPntCurvEmbPoly}
T.~Banchoff.
\newblock {C}ritical points and curvature for embedded polyhedra.
\newblock {\em J. Differential Geom.}, 1:245--256, 1967.

\bibitem{Banchoff83CritPtsCurvEmbPolyhII}
T.~F. Banchoff.
\newblock {C}ritical points and curvature for embedded polyhedra. {II}.
\newblock In {\em Differential geometry ({C}ollege {P}ark, {M}d., 1981/1982)},
  volume~32 of {\em Progr. Math.}, pages 34--55. Birkh\"auser Boston, Boston,
  MA, 1983.

\bibitem{Bieberbach12Bewegungsgruppen}
L.~Bieberbach.
\newblock \"{U}ber die {B}ewegungsgruppen der {E}uklidischen {R}\"aume
  ({Z}weite {A}bhandlung.) {D}ie {G}ruppen mit einem endlichen
  {F}undamentalbereich.
\newblock {\em Math. Ann.}, 72(3):400--412, 1912.

\bibitem{Bjoerner00SimplMnfBistellarFlips}
A.~Bj{\"o}rner and F.~H. Lutz.
\newblock {S}implicial manifolds, bistellar flips and a 16-vertex triangulation
  of the {P}oincar{\'e} homology 3-sphere.
\newblock {\em Experiment. Math.}, 9(2):275--289, 2000.

\bibitem{Brehm09LatticeTrigE33Torus}
U.~Brehm and W.~K{\"u}hnel.
\newblock Lattice triangulations of {$E^3$} and of the $3$-torus.
\newblock {\em {Israel J. Math.}}, 189:97--133, 2012.

\bibitem{Burton07EnumNonOr3Mflds}
B.~A. Burton.
\newblock Enumeration of non-orientable 3-manifolds using face-pairing graphs
  and union-find.
\newblock {\em Discrete Comput. Geom.}, 38(3):527--571, 2007.

\bibitem{Burton09Regina}
B.~A. Burton, R.~Budney, W.~Pettersson, et~al.
\newblock Regina: normal surface and 3-manifold topology software, version
  4.93.
\newblock {\tt http://\allowbreak regina.\allowbreak sourceforge.\allowbreak
  net/}, 1999--2012.

\bibitem{Spreer12VarCyclicPolytopeCompExpI}
B.~A.~Burton and J.~Spreer.
\newblock Combinatorial Seifert Fibred spaces with transitive cyclic
automorphism group, 2013.
\newblock In preparation, 23 pages, 7 figures.

\bibitem{simpcompISSAC}
F.~Effenberger and J.~Spreer.
\newblock simpcomp - a {GAP} toolbox for simplicial complexes.
\newblock {\em ACM Communications in Computer Algebra}, 44(4):186 -- 189, 2010.

\bibitem{simpcomp}
F.~Effenberger and J.~Spreer.
\newblock simpcomp - a {GAP} package, {V}ersion 1.6.1.
\newblock \url{http://www.igt.uni-stuttgart.de/LstDiffgeo/simpcomp}, 2013.

\bibitem{simpcompISSAC11}
F.~Effenberger and J.~Spreer.
\newblock Simplicial blowups and discrete normal surfaces in the {GAP} package
  simpcomp.
\newblock {\em ACM Communications in Computer Algebra}, 45(3):173 -- 176, 2011.

\bibitem{Emch29TripleQuadrSystems}
A.~Emch.
\newblock Triple and multiple systems, their geometric configurations and
  groups.
\newblock {\em Trans. Amer. Math. Soc.}, 31(1):25--42, 1929.

\bibitem{GAP4}
{GAP -- Groups, Algorithms, and Programming}, {Version 4.5.6}.
\newblock \url{http://www.gap-system.org}, 2012.

\bibitem{Hempel1976}
J.~Hempel.
\newblock $3$-{M}anifolds.
\newblock {\em Annals of Mathematics Studies}, pages 24--26, 1976.

\bibitem{Kuehnel95TightPolySubm}
W.~K{\"u}hnel.
\newblock {\em {T}ight polyhedral submanifolds and tight triangulations},
  volume 1612 of {\em Lecture Notes in Math.}
\newblock Springer, 1995.

\bibitem{Kuhnel96CentrSymmTightSurf}
W.~K{\"u}hnel.
\newblock {C}entrally-symmetric tight surfaces and graph embeddings.
\newblock {\em Beitr{\"a}ge Algebra Geom.}, 37(2):347--354, 1996.

\bibitem{Kuehnel98TopAsp2foldTripSys}
W.~K{\"u}hnel.
\newblock {T}opological aspects of twofold triple systems.
\newblock {\em Exposition. Math.}, 16(4):289--332, 1998.

\bibitem{Kuehnel02DiffGeom}
W.~K{\"u}hnel.
\newblock {\em Differential geometry}, volume~16 of {\em Student Mathematical
  Library}.
\newblock American Mathematical Society, Providence, RI, 2002.
\newblock Curves---surfaces---manifolds, Translated from the 1999 German
  original by Bruce Hunt.

\bibitem{Kuehnel85NeighbComb3MfldsDihedralAutGroup}
W.~K{\"u}hnel and G.~Lassmann.
\newblock {N}eighborly combinatorial {$3$}-manifolds with dihedral automorphism
  group.
\newblock {\em Israel J. Math.}, 52(1-2):147--166, 1985.

\bibitem{Kuhnel88CombDToriLargeSymm}
W.~K{\"u}hnel and G.~Lassmann.
\newblock {C}ombinatorial {$d$}-tori with a large symmetry group.
\newblock {\em Discrete Comput. Geom.}, 3(2):169--176, 1988.

\bibitem{Kuehnel96PermDiffCyc}
W.~K{\"u}hnel and G.~Lassmann.
\newblock {P}ermuted difference cycles and triangulated sphere bundles.
\newblock {\em Discrete Math.}, 162(1-3):215--227, 1996.

\bibitem{Kuiper71MorseRelations}
N.~H. Kuiper.
\newblock {M}orse relations for curvature and tightness.
\newblock In {\em Proceedings of {L}iverpool {S}ingularities {S}ymposium, {II}
  (1969/1970)}, volume 209 of {\em Lecture Notes in Math.}, pages 77--89,
  Berlin, 1971.

\bibitem{Lindner80CyclicSteinerSystems}
C.~C. Lindner and A.~Rosa, editors.
\newblock {\em Topics on {S}teiner systems}, volume~7 of {\em Ann. Discrete
  Math}.
\newblock North-Holland Publishing Co., Amsterdam, 1980.

\bibitem{Luft843MfldsWithSubgroupsContainingZ3}
E.~Luft and D.~Sjerve.
\newblock {$3$}-manifolds with subgroups {${\bf Z}\oplus {\bf Z}\oplus {\bf
  Z}$} in their fundamental groups.
\newblock {\em Pacific J. Math.}, 114(1):191--205, 1984.

\bibitem{Lutz08ManifoldPage}
F.~H. Lutz.
\newblock {T}he {M}anifold {P}age.
\newblock {\url{http://page.math.tu-berlin.de/~lutz/stellar/}}.

\bibitem{Lutz11TrigMnflds}
F.~H. Lutz.
\newblock Triangulating manifolds.
\newblock In press, ISBN 978-3-540-34502-2.

\bibitem{Lutz03TrigMnfFewVertVertTrans}
F.~H. Lutz.
\newblock {\em {T}riangulated manifolds with few vertices and vertex-transitive
  group actions}.
\newblock PhD thesis, TU Berlin, Aachen, 1999.

\bibitem{Lutz09EquivdCovTrigsSurf}
F.~H. Lutz.
\newblock {E}quivelar and $d$-covered triangulations of surfaces. {II}.
  {C}yclic triangulations and tessellations.
\newblock {\tt arXiv:1001.2779v1 [math.CO]}, 2010.
\newblock {To appear in Contrib. Discr. Math.}

\bibitem{Lutz08FVec3Mnf}
F.~H. Lutz, T.~Sulanke, and E.~Swartz.
\newblock $f$-vectors of $3$-manifolds.
\newblock {\em Electron. J. Comb.}, 16(2):{Research Paper 13, 33}, 2009.

\bibitem{Matveev13Recognizer}
S.~Matveev et al.
\newblock {{T}hree-manifold {R}ecognizer}.
\newblock \url{http://matlas.math.csu.ru/}, 2013.

\bibitem{Milnor03TowPoincareClass3Mnf}
J.~Milnor.
\newblock {T}owards the {P}oincar{\'e} conjecture and the classification of
  3-manifolds.
\newblock {\em Notices Amer. Math. Soc.}, 50(10):1226--1233, 2003.

\bibitem{Orlik72SeifertMflds}
P.~Orlik.
\newblock {\em Seifert manifolds}.
\newblock Lecture Notes in Mathematics, Vol. 291. Springer,
  1972.

\bibitem{Perelman03PC1}
G.~Perelman.
\newblock The entropy formula for the ricci flow and its geometric
  applications.
\newblock \texttt{arXiv:math.DG/0211159}, 2002.

\bibitem{Perelman03PC3}
G.~Perelman.
\newblock Finite extinction time for the solutions to the ricci flow on certain
  three-manifolds.
\newblock \texttt{arXiv:math.DG/0307245}, 2003.

\bibitem{Perelman03PC2}
G.~Perelman.
\newblock Ricci flow with surgery on three-manifolds.
\newblock \texttt{arXiv:math.DG/0303109}, 2003.

\bibitem{Ringel74MapColThm}
G.~Ringel.
\newblock {\em {M}ap color theorem}, volume 209 of {\em Die Grundlehren der
  mathematischen Wissenschaften}.
\newblock Springer, 1974.

\bibitem{Rubinstein953SphereRec}
J.~H. Rubinstein.
\newblock An algorithm to recognize the {$3$}-sphere.
\newblock In {\em Proceedings of the International Congress of Mathematicians
  ({Z}{\"u}rich, 1994)}, volume~1, pages 601--611. Birkh{\"a}user, 1995.

\bibitem{Saveliev02InvariantsOfHomology3Spheres}
N.~Saveliev.
\newblock {\em Invariants of Homology 3-Spheres}, volume 1 of 
  {\em Low-dimensional topology, Encyclopaedia of mathematical sciences}.
\newblock Springer, 2002.

\bibitem{Spreer10Diss}
J.~Spreer.
\newblock {\em Blowups, slicings and permutation groups in combinatorial
  topology}.
\newblock Logos Verlag Berlin, 2011.

\bibitem{Spreer10NormSurfsCombSlic}
J.~Spreer.
\newblock {N}ormal surfaces as combinatorial slicings.
\newblock {\em Discrete Math.}, 311(14):1295--1309, 2011.
\newblock {\tt doi:10.1016/j.disc.2011.03.013}.

\bibitem{Stallings59KnesersConj}
J.~R. Stallings.
\newblock {\em Some topological proofs and extensions of {G}rushko's theorem}.
\newblock PhD thesis, Princeton University, 1959.

\bibitem{Threlfall33SphaerischeMgftkten}
W.~Threlfall and H.~Seifert.
\newblock {Topologische Untersuchung der Diskontinuit\"atsbereiche endlicher
  Bewegungsgruppen des dreidimensionalen sph\"arischen Raumes, Schlu \ss}.
\newblock {\em Math. Ann.}, 107:543--586, 1933.

\bibitem{Thurston02GeomTopOf3Mflds}
W.~P. Thurston.
\newblock {\em The geometry and topology of 3-manifolds}, volume~1.
\newblock {Princeton University Press}, Princeton, N.J., 1980.
\newblock Electronic version 1.1 - March 2002.

\bibitem{Turaev92TuraevViroInvariant}
V.~G. Turaev and O.~Y. Viro.
\newblock State sum invariants of {$3$}-manifolds and quantum {$6j$}-symbols.
\newblock {\em Topology}, 31(4):865--902, 1992.

\bibitem{WeeksSnapPea}
J.~Weeks.
\newblock {{S}nap{P}y ({S}oftware for hyperbolic $3$-manifolds)}.
\newblock \url{http://www.math.uic.edu/t3m/SnapPy/}, 1999.

\end{thebibliography}
\end{document}